\documentclass[a4paper]{amsart}

\usepackage[english]{babel}
\usepackage[utf8x]{inputenc}
\usepackage[T1]{fontenc}
\usepackage{mathtools}
\usepackage{bm}

\usepackage{etoolbox}
\makeatletter
\preto{\@verbatim}{\topsep=0pt \partopsep=0pt }
\makeatother
\usepackage{url}

\usepackage{amsmath}
\usepackage{graphicx}
\usepackage[colorinlistoftodos]{todonotes}
\usepackage[colorlinks=true, allcolors=blue]{hyperref}
\usepackage{float}
\usepackage{fancyhdr}
\usepackage{amssymb}
\usepackage{multicol}
\usepackage{dsfont}
\usepackage{microtype}
\usepackage{listings}
\usepackage{enumerate}
\usepackage{amsthm}
\usepackage{amsmath}
\usepackage{tikz}
\usepackage{pgfplots}
\pgfplotsset{compat=1.11}
\usepgfplotslibrary{fillbetween}
\usetikzlibrary{intersections}
\pgfdeclarelayer{bg}
\pgfsetlayers{bg,main}
\usepackage{url}
\usepackage{hyperref}
\usepackage{amssymb}
\usepackage{MnSymbol}
\usepackage{tikz}
\usepackage{xcolor}
\usepackage{pgfkeys}
\usepackage{caption}
\usetikzlibrary{}

\usepackage[all,cmtip]{xy}

\theoremstyle{plain}
\newtheorem{theorem}{Theorem}[section]
\newtheorem*{theorem*}{Theorem}
\newtheorem{lemma}[theorem]{Lemma}
\newtheorem{proposition}[theorem]{Proposition}
\newtheorem{corollary}[theorem]{Corollary}
\newtheorem{definition}[theorem]{Definition}

\newtheorem{rem}[theorem]{Remark}
\newtheorem*{thma}{Theorem A}
\newtheorem*{thmb}{Theorem B}

\theoremstyle{remark}

\newtheorem*{claim}{Claim}

	\pgfkeys{
		/churro/.is family
		, /churro
		, default/.style = 
		{ scale=1
		, x=0
		, y=0
		, color=white
		}
		, scale/.estore in=\Churroscale 
		, x/.estore in=\Churrox
		, y/.estore in=\Churroy
		, color/.estore in=\Churrocolor
	}
	\newcommand{\churro}[1][]{		
		\pgfkeys{/churro, default, #1} 
		\begin{scope}[scale=\Churroscale,xshift=\Churrox*28.45/\Churroscale pt, yshift=\Churroy*28.45/\Churroscale pt] 
			\draw [fill=\Churrocolor] 
			(0,0) to[out=0,in=-90] (0.3,0.3)
			to[out=90,in=-120] (0.2,1.25)
			to[out=60,in=120] (0.8,1.25)
			to[out=-60,in=90] (0.7,0.3)
			to[out=-90,in=180] (1,0);
			
			\draw 
			(0.3,1.1) to[out=-80,in=-100] (0.7,1.1);
			
			\draw 
			(0.35,1.06) to[out=70,in=110] (0.65,1.06);
		\end{scope}
	}
	
\usetikzlibrary{decorations.pathreplacing}
\usetikzlibrary{decorations.markings}

\newcommand{\id}{\operatorname{id}}
\newcommand{\im}{\operatorname{im}}

\newcommand{\rk}{\operatorname{rk}}
\newcommand{\Aut}{\operatorname{Aut}}

\newcommand{\colim}{\operatorname{colim}}

\newcommand{\coker}{\operatorname{coker}}
\newcommand{\Diff}{\operatorname{Diff}}

\newcommand{\sSet}{\operatorname{\mathsf{sSet}}}
\newcommand{\Alg}{\operatorname{Alg}}

\newcommand{\proj}{\operatorname{proj}}

\newcommand{\Arf}{\operatorname{Arf}}
\newcommand{\sMod}{\operatorname{\mathsf{sMod}}}

\newcommand{\R}{\mathbf{R}}
\newcommand{\Rq}{\mathbf{R^q}}

\newcommand{\gr}{\operatorname{gr}}
\newcommand{\Sp}{\mathsf{Sp}}

\newcommand{\F}{\operatorname{\mathbb{F}_{\ell}}}

\renewcommand{\mod}{\operatorname{mod}}
\newcommand{\FF}{\operatorname{\mathbb{F}_2}}

\newcommand{\G}{\mathsf{G}}

\newcommand{\kk}{\mathds{k}}
\newcommand{\X}{\mathbf{X}}
\newcommand{\Ak}{\mathbf{A_{\kk}}}
\newcommand{\Gq}{\mathsf{G^q}}

\newcommand{\MCG}{\mathsf{MCG}}
\newcommand{\Stab}{\operatorname{Stab}}

\newcommand{\AFF}{\operatorname{\mathbf{A_{\FF}}}}

\newcommand*{\Scale}[2][4]{\scalebox{#1}{$#2$}}%

\title[spin mapping class groups and quadratic symplectic groups]{Homological stability of spin mapping class groups and quadratic symplectic groups}

\subjclass[2010]{55R40, 57S05, 20J06}
\keywords{Homological stability, spin mapping class groups, symplectic groups, $E_k$-algebras}
\author{Ismael Sierra}

\begin{document}
\large
\maketitle

\begin{abstract}
We study the homological stability of spin mapping class groups of surfaces and of quadratic symplectic groups using cellular $E_2$-algebras. 
We get improvements in their stability results, which for the spin mapping class groups we show to be optimal away from the prime $2$. 
We also prove that in both cases the $\FF$-homology satisfies secondary homological stability.
Finally, we give full descriptions of the first homology groups of the spin mapping class groups and of the quadratic symplectic groups. 
\end{abstract}

\section{Introduction}
\subsection{Definition of the groups} \label{section intro}

We denote by $\Sigma_{g,1}$ the orientable surface of genus $g$ with one boundary component, and by $\Gamma_{g,1}=\pi_0(\Diff_{\partial}(\Sigma_{g,1}))$ its \textit{mapping class group}, defined to be the group of isotopy classes of diffeomorphisms of $\Sigma_{g,1}$ fixing pointwise a neighbourhood of its boundary. 
We will define the spin mapping class groups using the approach of \cite{harerspin}, which is based on the notion of quadratic refinements. 

Given an integer-valued skew-symmetric bilinear form $(M,\lambda)$ on a finitely generated free $\mathbb{Z}$-module, a \textit{quadratic refinement} is a function $q: M \rightarrow \mathbb{Z}/2$ such that $q(x+y) \equiv q(x)+q(y)+ \lambda(x,y) (\mod 2)$ for all $x,y \in M$. 
There are $2^{\rk(M)}$ quadratic refinements since a quadratic refinement is uniquely determined by its values on a basis of $M$, and any set of values is possible. 

The set of quadratic refinements of $(H_1(\Sigma_{g,1};\mathbb{Z}),\cdot)$ has a right $\Gamma_{g,1}$-action by precomposition. 
By \cite[Corollary 2]{Johnsonspin} this action has precisely two orbits for $g \geq 1$, distinguished by the \textit{Arf invariant}, which is a $\mathbb{Z}/2$-valued function on the set of quadratic refinements, see Definition \ref{defn arf}. 
For $\epsilon \in \{0,1\}$ we will denote by $\Gamma_{g,1}^{1/2}[\epsilon]:= \Stab_{\Gamma_{g,1}}(q)$ where $q$ is a choice of quadratic refinement of Arf invariant $\epsilon$, and call this the \textit{spin mapping class group in genus $g$ and Arf invariant $\epsilon$}. 

Similarly, for $g \geq 1$ the group $Sp_{2g}(\mathbb{Z})$ acts on the set of quadratic refinements of the standard symplectic form $(\mathbb{Z}^{2g},\Omega_g)$ with precisely two orbits, also distinguished by the Arf invariant. 
Thus, for $\epsilon \in \{0,1\}$ we can define the \textit{quadratic symplectic group in genus $g$ and Arf invariant $\epsilon$} to be $Sp_{2g}^{\epsilon}(\mathbb{Z}):= \Stab_{Sp_{2g}(\mathbb{Z})}(q)$ for a fixed quadratic refinement $q$ of Arf invariant $\epsilon$. 

\subsection{Statement of results} \label{section results}

Before stating the results let us recall what \textit{stabilization maps} mean in this context.
We begin by fixing quadratic refinements $q_0,q_1$ of $(H_1(\Sigma_{1,1};\mathbb{Z}),\cdot) \cong (\mathbb{Z}^2,\Omega_1)$ of Arf invariants $0,1$ respectively. 
Then, given any quadratic refinement $q$ of $(H_1(\Sigma_{g-1,1};\mathbb{Z}),\cdot) \cong (\mathbb{Z}^{2(g-1)},\Omega_{g-1})$ we get a quadratic refinement $q \oplus q_{\epsilon}$ of $(H_1(\Sigma_{g,1};\mathbb{Z}),\cdot) \cong (\mathbb{Z}^{2g},\Omega_{g})$. 
Moreover, the Arf invariant is additive, see Definition \ref{defn arf}, so $\Arf(q\oplus q_{\epsilon})=\Arf(q)+\epsilon$. 

Thus, using the inclusions $\Gamma_{g-1,1} \subset \Gamma_{g,1}$ and $Sp_{2(g-1)}(\mathbb{Z}) \subset Sp_{2g}(\mathbb{Z})$ we get \textit{stabilization maps} 
$$\Gamma_{g-1,1}^{1/2}[\delta-\epsilon] \rightarrow \Gamma_{g,1}^{1/2}[\delta]$$
and 
$$Sp_{2(g-1)}^{\delta-\epsilon}(\mathbb{Z}) \rightarrow Sp_{2g}^{\delta}(\mathbb{Z}).$$
The goal of this paper is to study homological stability with respect to these stabilisation maps. 
Before moving to the results observe that additivity of the Arf invariant under direct sum of quadratic refinements also allows us to define products 
$$\Gamma_{g,1}^{1/2}[\epsilon] \times \Gamma_{g',1}^{1/2}[\epsilon'] \rightarrow \Gamma_{g+g',1}^{1/2}[\epsilon+\epsilon']$$
and
$$Sp_{2g}^{\epsilon}(\mathbb{Z}) \times Sp_{2g'}^{\epsilon}(\mathbb{Z})\rightarrow Sp_{2(g+g')}^{\epsilon+\epsilon'}(\mathbb{Z})$$
which contain the stabilisation maps as particular cases. 

It is known since \cite[Theorem 3.1]{harerspin} that spin mapping class groups satisfy homological stability in the range $d \lesssim g/4$, and their stable homology can be understood by \cite[Section 1]{stablespin}. 
Thus, improvements in the stability range are important as they lead to new homology computations. 
In this direction, the previously known best bounds can be found in \cite[Theorem 2.14]{rspin}, where a range of the form $d \lesssim 2g/5$ was shown. 
The first main result of this paper improves the known stability range. 

\begin{thma} \label{theorem A}
Consider the stabilization map 
$$H_d(\Gamma_{g-1,1}^{1/2}[\delta-\epsilon];\mathds{k}) \rightarrow H_d(\Gamma_{g,1}^{1/2}[\delta];\mathds{k}),$$
then:
\begin{enumerate}[(i)]
    \item If $\mathds{k}=\mathbb{Z}$ it is surjective for $2d \leq g-2$ and an isomorphism for $2d \leq g-4$. 
    \item If $\mathds{k}=\mathbb{Z}[\frac{1}{2}]$ it is surjective for $3d \leq 2g-4$ and an isomorphism for $3d \leq 2g-7$.  
\end{enumerate}
Moreover, there is a homology class $\theta \in H_2(\Gamma_{4,1}^{1/2}[0];\FF)$ such that 
$$\theta \cdot-: H_{d-2}(\Gamma_{g-4,1}^{1/2}[\delta],\Gamma_{g-5,1}^{1/2}[\delta-\epsilon];\FF) \rightarrow H_d(\Gamma_{g,1}^{1/2}[\delta],\Gamma_{g-1,1}^{1/2}[\delta-\epsilon];\FF)$$ 
is surjective for $3d \leq 2g-5$ and an isomorphism for $3d \leq 2g-8$. 
\end{thma}

The result with $\mathbb{Z}[1/2]$-coefficients is optimal (up to possibly a better constant term) by Lemma \ref{lem optimallity}, and in particular the ``slope $2/3$'' cannot be improved. 
The last part of the theorem is an example of \textit{secondary homological stability}, which means that it gives a range in which the defects of homological stability are themselves stable. 
By Corollary \ref{cor 2 torsion}, a consequence is that the $\FF$-homology satisfies a $2/3$ slope stability if and only if $\theta^3$ destabilizes by $\sigma_{\epsilon}$; and otherwise the slope $1/2$ of part (i) would be optimal with $\FF$-coefficients, and hence integrally. 
We do not know which of these two alternatives holds. 
Finally we remark that the class $\theta$ is not uniquely defined, see Remarks \ref{rem indeterminacy} and \ref{rem theta well-defined}, but its indeterminacy is small and the statement above holds for any such choice of $\theta$. 

The second main result is about homological stability of quadratic symplectic groups. 

\begin{thmb} \label{theorem B}
Consider the stabilization map 
$$H_d(Sp_{2(g-1)}^{\delta-\epsilon}(\mathbb{Z});\mathds{k}) \rightarrow H_d(Sp_{2g}^{\delta}(\mathbb{Z});\mathds{k}),$$
then:
\begin{enumerate}[(i)]
    \item If $\mathds{k}=\mathbb{Z}$ it is surjective for $2d \leq g-2$ and an isomorphism for $2d \leq g-4$. 
    \item If $\mathds{k}=\mathbb{Z}[\frac{1}{2}]$ it is surjective for $3d \leq 2g-4$ and an isomorphism for $3d \leq 2g-7$.  
\end{enumerate} 
Moreover, there is a homology class $\theta \in H_2(Sp_8^0(\mathbb{Z});\FF)$ such that 
$$\theta \cdot-: H_{d-2}(Sp_{2(g-4)}^{\delta}(\mathbb{Z}),Sp_{2(g-5)}^{\delta-\epsilon}(\mathbb{Z});\FF) \rightarrow H_d(Sp_{2g}^{\delta}(\mathbb{Z}),Sp_{2(g-1)}^{\delta-\epsilon}(\mathbb{Z});\FF)$$ 
is surjective for $3d \leq 2g-5$ and an isomorphism for $3d \leq 2g-8$. 
\end{thmb}

The groups $Sp_{2g}^{0}(\mathbb{Z})$ have appeared in the literature under the name of \textit{theta subgroups of the symplectic groups}, and sometimes denoted by $Sp_{2g}^{q}(\mathbb{Z})$. 
These groups are of importance in number theory, see \cite{presentationsymplectic} for example, and in the study of manifolds, as in \cite[Section 4]{framings}. 
The groups $Sp_{2g}^{1}(\mathbb{Z})$ are less common but have appeared recently in the study of manifolds in \cite[Section 4]{framings}, where they are denoted by $Sp_{2g}^{a}(\mathbb{Z})$. 

Some results were previously known about homological stability of quadratic symplectic groups. 
In particular, \cite[Theorem 5.2]{Nina} already gave a stability result of the form $d \lesssim g/2$ following different techniques. 
However, the improvement to $d \lesssim 2g/3$ in part (ii) of the above theorem is new. 
As before, the last part is a secondary stability result which implies that either the $\FF$-homology also has slope $2/3$ stability (if $\theta^3$ destabilises) or the optimal slope of the $\FF$-homology is $1/2$ (otherwise).
The class $\theta$ is again not well-defined but its indeterminacy is understood by Remarks \ref{rem indeterminacy} and \ref{rem theta well-defined}. 

We will prove Theorems \hyperref[theorem A]{A} and \hyperref[theorem B]{B} using the technique of \textit{cellular $E_k$-algebras} developed in \cite{Ek}, and in particular we follow some ideas of \cite{E2} where this approach is applied to homological stability of mapping class groups of surfaces. 

The basic idea is to define $E_2$-algebra structures on both $\bigsqcup_{g,\epsilon} B \Gamma_{g,1}^{1/2}[\epsilon]$ and $\bigsqcup_{g,\epsilon} B \Sp_{2g}^{\epsilon}(\mathbb{Z})$ which are ``induced by the products'' 
$$\Gamma_{g,1}^{1/2}[\epsilon] \times \Gamma_{g',1}^{1/2}[\epsilon'] \rightarrow \Gamma_{g+g',1}^{1/2}[\epsilon+\epsilon']$$
and 
$$Sp_{2g}^{\epsilon}(\mathbb{Z}) \times Sp_{2g'}^{\epsilon}(\mathbb{Z})\rightarrow Sp_{2(g+g')}^{\epsilon+\epsilon'}(\mathbb{Z})$$
respectively. 
This structure contains the stabilisation maps but also captures more information, which will be used to prove the above stability ranges and to properly define the class $\theta$ and the secondary stabilisation. 

\subsection{Overview of cellular $E_2$-algebras} \label{section E2 algebras overview}
The purpose of this section is to explain the methods from \cite{Ek} used in this paper: we aim for an informal discussion and refer to \cite{Ek} for details.

In the $E_2$-algebras part of the paper we will work in the category $\mathsf{sMod}_{\mathds{k}}^{\mathsf{G}}$ of $\mathsf{G}$-graded simplicial $\mathds{k}$-modules, for $\mathds{k}$ a commutative ring and $\mathsf{G}$ a discrete symmetric monoid. 
Formally, $\mathsf{sMod}_{\mathds{k}}^{\mathsf{G}}$ denotes the category of functors from $\mathsf{G}$, viewed as a category with objects the elements of $\mathsf{G}$ and only identity morphisms, to $\mathsf{sMod}_{\mathds{k}}$. 
This means that each object $M$ consists of a simplicial $\mathds{k}$-module $M_{\bullet}(x)$ for each $x \in \mathsf{G}$. 
The tensor product $\otimes$ in this category is given by Day convolution, i.e.
$$ (M \otimes N)_p(x)= \bigoplus_{y+z=x}{M_p(y) \otimes_{\mathds{k}} N_p(z)}$$
where $+$ denotes the monoidal structure of $\mathsf{G}$. 

In a similar way one can define the category of $\mathsf{G}$-graded spaces, denoted by $\mathsf{Top}^{\mathsf{G}}$ and endow it with a monoidal structure by Day convolution using cartesian product of spaces. 

The \textit{little 2-cubes} operad in $\mathsf{Top}$ has $n$-ary operations given by rectilinear embeddings $I^2 \times \{1,\cdots,n\} \hookrightarrow I^2$ such that the interiors of the images of the $2$-cubes are disjoint.
(The space of 0-ary operations is empty.) 
We define the little $2$-cubes operad in $\mathsf{sMod}_{\mathds{k}}$ by applying the symmetric monoidal functor
$(-)_{\mathds{k}}: \mathsf{Top} \rightarrow \mathsf{sMod}_{\mathds{k}}$ given by the free $\mathds{k}$-module on the singular simplicial set of a space.
Moreover, $(-)_{\mathds{k}}$ can be promoted to a functor $(-)_{\mathds{k}}: \mathsf{Top}^{\mathsf{G}} \rightarrow \mathsf{sMod}_{\mathds{k}}^{\mathsf{G}}$ between the graded categories, and we define the little $2$-cubes operad in these by concentrating it in grading $0$, where $0 \in \mathsf{G}$ denotes the identity of the monoid.  
We shall denote this operad by $\mathcal{C}_2$ in all the categories  $\mathsf{Top}, \mathsf{Top}^{\mathsf{G}}, \mathsf{sMod}_{\mathds{k}}^{\mathsf{G}}$ which we use, and define an $E_2$-\textit{algebra} to mean an algebra over this operad. 

The $E_2$-\textit{indecomposables} of an $E_2$-algebra $\textbf{R}$ in $\mathsf{sMod}_{\mathds{k}}^{\mathsf{G}}$ is defined by the exact sequence of graded simplicial $\mathds{k}$-modules

$$ \bigoplus_{n \geq 2}{\mathcal{C}_2(n) \otimes \textbf{R}^{\otimes n}} \rightarrow \textbf{R} \rightarrow Q^{E_2}(\textbf{R}) \rightarrow 0.$$

The functor $\textbf{R} \mapsto Q^{E_2}(\textbf{R})$ is not homotopy-invariant but has a derived functor $Q_{\mathbb{L}}^{E_2}(-)$ which is.
See \cite[Section 13]{Ek} for details and how to define it in more general categories such as $E_2$-algebras in $\mathsf{Top}$ or $\mathsf{Top}^{\mathsf{G}}$. 
The $E_2$-\textit{homology} groups of $\textbf{R}$ are defined to be
$$H_{x,d}^{E_2}(\textbf{R}):=H_d(Q_{\mathbb{L}}^{E_2}(\textbf{R})(x))$$
for $x \in \mathsf{G}$ and $d \in \mathbb{N}$. 
One formal property of the derived indecomposables, see \cite[Lemma 18.2]{Ek}, is that it commutes with $(-)_{\mathds{k}}$, so for $\textbf{R} \in \Alg_{E_2}(\mathsf{Top}^{\mathsf{G}})$ its $E_2$-homology with $\mathds{k}$ coefficients is the same as the $E_2$-homology of $\textbf{R}_{\mathds{k}}$. 

Thus, in order to study homological stability of $\textbf{R}$ with different coefficients we can work with the $E_2$-algebras $\textbf{R}_{\mathds{k}}$ instead, which enjoy better properties as they are cofibrant and the category of graded simplicial $\mathds{k}$-modules offers some technical advantages as explained in \cite[Section 11]{Ek}.
However, at the same time, we can do computations in $\mathsf{Top}$ of the homology or $E_2$-homology of $\textbf{R}$ and then transfer them to $\mathsf{sMod}_{\mathds{k}}$. 

In \cite[Section 6]{Ek} the notion of a \textit{CW $E_2$-algebra} is defined, built in terms of free $E_2$-algebras by iteratively attaching cells in the category of $E_2$-algebras in order of dimension. 

Let $\Delta^{x,d} \in \mathsf{sSet}^{\mathsf{G}}$ be the standard $d$-simplex placed in grading $x$ and let $\partial \Delta^{x,d} \in \mathsf{sSet}^{\mathsf{G}}$ be its boundary. 
By applying the free $\mathds{k}$-module functor we get objects $\Delta_{\mathds{k}}^{x,d}, \partial \Delta_{\mathds{k}}^{x,d} \in \mathsf{sMod}_{\mathds{k}}^{\mathsf{G}}$. 
We then define the graded spheres in $\mathsf{sMod}_{\mathds{k}}^{\mathsf{G}}$ via $S_{\mathds{k}}^{x,d}:=\Delta_{\mathds{k}}^{x,d}/ \partial \Delta_{\mathds{k}}^{x,d}$, where the quotient denotes the cofibre of the inclusion of the boundary into the disc. 
In $\mathsf{sMod}_{\mathds{k}}^{\mathsf{G}}$, the data for a cell attachment to an $E_2$-algebra $\textbf{R}$ is an \textit{attaching map} $e: \partial \Delta_{\mathds{k}}^{x,d} \rightarrow \textbf{R}$, which is the same as a map $\partial \Delta^d_{\mathds{k}} \rightarrow \textbf{R}(x)$ of simplicial $\mathds{k}$-modules. 
To attach the cell we form the pushout in $\Alg_{E_2}(\mathsf{sMod}_{\mathds{k}}^{\mathsf{G}})$

\centerline{\xymatrix{ \mathbf{E_2}(\partial \Delta_{\mathds{k}}^{x,d}) \ar[r] \ar[d] & \textbf{R} \ar[d] \\
\mathbf{E_2}(\Delta_{\mathds{k}}^{x,d}) \ar[r] & \textbf{R} \cup_{e}^{E_2} D_{\mathds{k}}^{x,d}.
}}

A weak equivalence $\textbf{C} \xrightarrow{\sim} \textbf{R}$ from a CW $E_2$-algebra is called a \textit{CW-approximation} to
$\textbf{R}$, and a key result, \cite[Theorem 11.21]{Ek}, is that if $\textbf{R}(0) \simeq 0$ then $\textbf{R}$ admits a CW-approximation. 
Moreover, whenever $\mathds{k}$ is a field, we can construct a CW-approximation in which the number of $(x,d)$-cells needed is precisely the dimension of $H_{x,d}^{E_2}(\textbf{R})$.
By “giving the $d$-cells filtration $d$”, see \cite[Section 11]{Ek} for a more precise discussion of what this
means, one gets a skeletal filtration of this $E_2$-algebra and a spectral sequence computing the homology of $\textbf{R}$. 

In order to discuss homological stability of $E_2$-algebras we will need some preparation. 
For the rest of this section let $\textbf{R} \in \Alg_{E_2}(\mathsf{sMod}_{\mathds{k}}^{\mathsf{G}})$, where $\mathsf{G}$ is equipped with a symmetric monoidal functor $\rk: \mathsf{G} \rightarrow \mathbb{N}$; and suppose we are given a homology class $\sigma \in H_{x,0}(\textbf{R})$ with $\rk(x)=1$.
By definition $\sigma$ is a homotopy class of maps $\sigma: S_{\mathds{k}}^{x,0} \rightarrow \mathbf{R}$.

Following \cite[Section 12.2]{Ek}, there is a strictly associative algebra $\mathbf{\overline{R}}$ which is equivalent to the unitalization $\mathbf{R^+}:=\mathds{1} \oplus \textbf{R}$, where $\mathds{1}$ is the monoidal unit in simplicial modules. 
Then, $\sigma$ gives a map $\sigma \cdot-: S_{\mathds{k}}^{x,0} \otimes \mathbf{\overline{R}} \rightarrow \mathbf{\overline{R}}$ by using the associative product of $\mathbf{\overline{R}}$. 
We then define $\mathbf{\overline{R}}/\sigma$ to be the cofibre of this map.
Observe that a-priori $\sigma \cdot -$ is not a (left) $\mathbf{\overline{R}}$-module map, so the cofibre $\mathbf{\overline{R}}/\sigma$ is not a (left) $\mathbf{\overline{R}}$-module.
However, by the ``adapters construction'' in \cite[Section 12.2]{Ek} and its  applications in \cite[Section 12.2.3]{Ek}, there is a way of defining a cofibration sequence $S^{1,0} \otimes \mathbf{\overline{R}} \xrightarrow{\sigma \cdot - } \mathbf{\overline{R}} \rightarrow \mathbf{\overline{R}}/\sigma$ in the category of left $\mathbf{\overline{R}}$-modules in such a way that forgetting the $\mathbf{\overline{R}}$-module structure recovers the previous construction; we will make use of this fact in some of the proofs of Section \ref{section proof}. 

By construction $\sigma \cdot -$ induces maps $\textbf{R}(y) \rightarrow \textbf{R}(x+y)$ between the different graded components of $\textbf{R}$ and the homology of the object $\mathbf{\overline{R}}/\sigma$ captures the relative homology of these. 
Thus, homological stability results of $\textbf{R}$ using $\sigma$ to stabilize can be reformulated as vanishing ranges for $H_{x,d}(\mathbf{\overline{R}}/\sigma)$; the advantage of doing so is that using filtrations for CW approximations of $\textbf{R}$ one also gets filtrations of $\mathbf{\overline{R}}/\sigma$ and hence spectral sequences capable of detecting vanishing ranges. 

The secondary stability result can be written in terms of $E_2$-algebras in a similar way: 
this time we will have a class $\sigma$ as above and another class $\theta$, and we will prove a vanishing in the homology of the iterated cofibre construction $\R/(\sigma,\theta):=(\R/\sigma)/\theta$, in the sense of \cite[Section 12.2.3]{Ek}.

\subsection{Organization of the paper}

In Section \ref{section Ek} we will state generic stability results for $E_2$-algebras, which will be shown in Section \ref{section proof} and then used to prove Theorems \hyperref[theorem A]{A} and \hyperref[theorem B]{B}. 

In Section \ref{section 4} we will define the notion of ``quadratic data'' and explain how it produces a ``quadratic $E_2$-algebra''. 
This construction generalizes the way that spin mapping class groups and quadratic symplectic groups are defined from the mapping class groups and symplectic groups respectively. 
Finally, Theorem \ref{theorem splitting complexes} and Corollary \ref{cor std connectivity} will give ways to access information about the $E_2$-cells of the quadratic $E_2$-algebra from knowledge about the underlying non-quadratic algebra. 

Section \ref{section symplectic} is devoted to the proof of Theorem \hyperref[theorem B]{B}, which is an application of the results of the previous sections. 
In the proof of the last two parts of the theorem we will also need some information about the first homology groups of quadratic symplectic groups, which can be found in the \hyperref[appendix]{Appendix}. 

Section \ref{section mcg} contains the proof of Theorem \hyperref[theorem A]{A}, which follows similar steps to the previous section. 

Finally, the \hyperref[appendix]{Appendix} contains detailed computations of the first homology groups of spin mapping class groups and quadratic symplectic groups. 
Let us remark that a full description of all first homology groups and stabilization maps is included for completeness, even if not everything there is used to prove Theorems \hyperref[theorem A]{A} and \hyperref[theorem B]{B}. 
The main idea behind the computations is to start with known presentations of the mapping class groups and symplectic groups and then to find presentations for the (finite index subgroups) spin mapping class groups and quadratic symplectic groups using GAP.

\section*{Acknowledgements.}
I am supported by an EPSRC PhD Studentship, grant no. 2261123, and by O. Randal-Williams’ Philip Leverhulme Prize from the Leverhulme Trust.
I would like to give special thanks to my PhD supervisor Oscar Randal-Williams for all his advice and all the helpful discussions and corrections. 

\section{Generic homological stability results} \label{section Ek}

In this section we will state three homological stability results for $E_2$-algebras, Theorems \ref{theorem stab 1}, \ref{theorem stab 2} and \ref{thm stab 3}, that will later apply to quadratic symplectic groups and spin mapping class groups in Sections \ref{section symplectic} and \ref{section mcg}. 
The first two of these are inspired by the generic homological stability theorem \cite[Theorem 18.1]{Ek}, in the sense that they input a vanishing line on the $E_2$-homology of an $E_2$-algebra along with some information about the homology in small bidegrees, and they output homological stability results for the algebra. 
The third one is a secondary stability result, which is inspired by \cite[Lemma 5.6, Theorem 5.12]{E2}. 

In addition, we have Corollary \ref{cor 2 torsion}, which says that $E_2$-algebras satisfying the assumptions of Theorem \ref{thm stab 3} have homological stability of slope either exactly $1/2$ or at least $2/3$ depending on the value of a certain homology class. 
However, the precise statement of this corollary is delayed to the next section until we have properly defined the secondary stabilisation map.

Before stating the results let us define the grading category that will be relevant:
let $\mathsf{H}$ be the discrete monoid $\{0\} \cup \mathbb{N}_{>0} \times \mathbb{Z}/2$, where the monoidal structure $+$ is given by addition in both coordinates. 
We denote by $\rk: \mathsf{H} \rightarrow \mathbb{N}$ the monoidal functor given by projection to the first coordinate. 

Also, let us recall that on $\Alg_{E_2}(\mathsf{sMod}_{\mathds{k}}^{\mathbb{N}})$ there is a homology operation $Q_{\mathds{k}}^1(-): H_{*,0}(-) \rightarrow H_{2*,1}(-)$ defined in \cite[Page 199]{Ek}.
This operation satisfies that $-2 Q_{\mathds{k}}^1(-)=[-,-]$, where $[-,-]$ is the Browder bracket. 
By using the canonical rank functor $\rk: \mathsf{H} \rightarrow \mathbb{N}$ we can view any $\mathsf{H}$-graded $E_2$-algebra as $\mathbb{N}$-graded and hence make sense of this operation on $\Alg_{E_2}(\mathsf{sMod}_{\mathds{k}}^{\mathsf{H}})$ too. 

\begin{theorem} \label{theorem stab 1}
Let $\kk$ be a commutative ring and let $\X \in \Alg_{E_2}(\sMod_{\kk}^{\mathsf{H}})$ be such that $H_{0,0}(\X)=0$, $H_{x,d}^{E_2}(\X)=0$ for $d<\rk(x)-1$, and $H_{*,0}(\mathbf{\overline{X}})=\frac{\kk[\sigma_0,\sigma_1]}{(\sigma_1^2-\sigma_0^2)}$ as a ring, for some classes $\sigma_{\epsilon} \in H_{(1,\epsilon),0}(\X)$. 
Then, for any $\epsilon \in \{0,1\}$ and any $x \in \mathsf{H}$ we have $H_{x,d}(\mathbf{\overline{X}}/\sigma_{\epsilon})=0$ for $2d \leq \rk(x)-2$. 
\end{theorem}

\begin{theorem} \label{theorem stab 2}
Let $\kk$ be a commutative $\mathbb{Z}[1/2]$-algebra, let $\textbf{X} \in \Alg_{E_2}(\mathsf{sMod}_{\mathds{k}}^{\mathsf{H}})$ be such that $H_{0,0}(\mathbf{X})=0$,  $H_{x,d}^{E_2}(\textbf{X})=0$ for $d<\rk(x)-1$, and  $H_{*,0}(\mathbf{\overline{X}})=\frac{\mathds{k}[\sigma_0,\sigma_1]}{(\sigma_1^2-\sigma_0^2)}$ as a ring, for some classes $\sigma_{\epsilon} \in H_{(1,\epsilon),0}(\textbf{X})$. 
Suppose in addition that for some $\epsilon \in \{0,1\}$ we have:
\begin{enumerate}[(i)]
    \item $\sigma_{\epsilon} \cdot - : H_{(1,1-\epsilon),1}(\textbf{X}) \rightarrow H_{(2,1),1}(\textbf{X})$ is surjective.
    \item $\coker(\sigma_{\epsilon} \cdot-: H_{(1,\epsilon),1}(\textbf{X}) \rightarrow H_{(2,0),1}(\textbf{X}))$ is generated by $Q_{\mathds{k}}^1(\sigma_0)$ as a $\mathbb{Z}$-module.
    \item $\sigma_{1-\epsilon} \cdot Q_{\mathds{k}}^1(\sigma_0) \in H_{(3,1-\epsilon),1}(\textbf{X})$ lies in the image of $\sigma_{\epsilon}^2 \cdot -:H_{(1,1-\epsilon),1}(\textbf{X}) \rightarrow H_{(3,1-\epsilon),1}(\textbf{X})$.
\end{enumerate}
Then $H_{x,d}(\mathbf{\overline{X}}/\sigma_{\epsilon})=0$ for $3d \leq 2 \rk(x)-4$. 
\end{theorem}

\begin{theorem} \label{thm stab 3}
Let $\textbf{X} \in \Alg_{E_2}(\mathsf{sMod}_{\FF}^{\mathsf{H}})$ be such that $H_{0,0}(\mathbf{X})=0$,  $H_{x,d}^{E_2}(\textbf{X})=0$ for $d<\rk(x)-1$, and  $H_{*,0}(\mathbf{\overline{X}})=\frac{\FF[\sigma_0,\sigma_1]}{(\sigma_1^2-\sigma_0^2)}$ as a ring, for some classes $\sigma_{\epsilon} \in H_{(1,\epsilon),0}(\textbf{X})$. 
Suppose in addition that for some $\epsilon \in \{0,1\}$ we have:
\begin{enumerate}[(i)]
    \item $\sigma_{\epsilon} \cdot - : H_{(1,1-\epsilon),1}(\textbf{X}) \rightarrow H_{(2,1),1}(\textbf{X})$ is surjective.
    \item $\coker(\sigma_{\epsilon} \cdot-: H_{(1,\epsilon),1}(\textbf{X}) \rightarrow H_{(2,0),1}(\textbf{X}))$ is generated by $Q_{\FF}^1(\sigma_0)$.
    \item $\sigma_{1-\epsilon} \cdot Q_{\FF}^1(\sigma_0) \in H_{(3,1-\epsilon),1}(\textbf{X})$ lies in the image of $\sigma_{\epsilon}^2 \cdot -:H_{(1,1-\epsilon),1}(\textbf{X}) \rightarrow H_{(3,1-\epsilon),1}(\textbf{X})$.
    \item $\sigma_0 \cdot Q_{\FF}^1(\sigma_0) \in H_{(3,0),1}(\textbf{X})$ lies in the image of $\sigma_{\epsilon}^2 \cdot -:H_{(1,0),1}(\textbf{X}) \rightarrow H_{(3,0),1}(\textbf{X})$.
\end{enumerate}
Then there is a class $\theta \in H_{(4,0),2}(\X)$ such that $H_{x,d}(\mathbf{\overline{X}}/(\sigma_{\epsilon},\theta))=0$ for $3d \leq 2 \rk(x)-5$.     
\end{theorem}

\section{Proving Theorems \ref{theorem stab 1}, \ref{theorem stab 2} and \ref{thm stab 3}} \label{section proof}

We will need some preparation. 
For $\kk$ a commutative ring we define 
$$\Ak:= \mathbf{E_2}(S_{\kk}^{(1,0),0} \sigma_0 \oplus S_{\kk}^{(1,1),0} \sigma_1) \cup_{\sigma_1^2-\sigma_0^2}^{E_2}{D_{\kk}^{(2,0),1} \rho} \in \Alg_{E_2}(\sMod_{\kk}^{\mathsf{H}}).$$

This should be thought of as a ``universal example'' in a sense that will be clear in Sections \ref{proof thm stab 1} and \ref{construction theta}.

\begin{rem}
To attach the $((2,0),1)$-cell $\rho$ we need a map $\partial \Delta_{\kk}^{(2,0),1} \rightarrow \mathbf{E_2}(S_{\kk}^{(1,0),0} \sigma_0 \oplus S_{\kk}^{(1,1),0} \sigma_1)$, whereas $\sigma_1^2-\sigma_0^2$ is a-priori a map $S_{\kk}^{(2,0),0} \rightarrow \mathbf{E_2}(S_{\kk}^{(1,0),0} \sigma_0 \oplus S_{\kk}^{(1,1),0} \sigma_1)$. 

The precise construction is as follows: for any $d$ there is a canonical map $\partial \Delta^{d+1}_{\kk} \rightarrow S^d_{\kk}$ which extends to the graded categories, hence allowing to make sense of cell attachments along maps defined on spheres. 
We will always use this slight abuse of notation for the rest of the paper without further mention. 
\end{rem}

\begin{proposition} \label{prop Ak}
The $E_2$-algebra $\Ak$ satisfies the assumptions of Theorem \ref{theorem stab 1}, i.e $H_{0,0}(\Ak)=0$, $H_{x,d}^{E_2}(\Ak)=0$ for $d<\rk(x)-1$ and $H_{*,0}(\mathbf{\overline{A_{\kk}}})=\kk[\sigma_0,\sigma_1]/(\sigma_1^2-\sigma_0^2)$ as a ring.
\end{proposition}

\begin{proof}
Since $\Ak$ is built using cells then $Q_{\mathbb{L}}^{E_2}(\Ak)=S_{\kk}^{(1,0),0} \sigma_0 \oplus S_{\kk}^{(1,1),0} \sigma_1 \oplus S_{\kk}^{(2,0),1} \rho$, see \cite[Sections 6.1.3 and 8.2.1]{Ek} for details.
Thus $H_{x,d}^{E_2}(\Ak)=0$ for $d<\rk(x)-1$. 

For the homology computations it suffices to consider the case $\kk=\mathbb{Z}$ by the following argument: 

Let us write $- \otimes \kk: \sMod_{\mathbb{Z}} \rightarrow \sMod_{\kk}$ for the base-change functor and for the corresponding functor between $\mathsf{H}$-graded categories. 
Base-change is symmetric monoidal, preserves colimits and satisfies $S_{\mathbb{Z}}^{x,d} \otimes \kk = S_{\kk}^{x,d}$, $\Delta_{\mathbb{Z}}^{x,d} \otimes \kk = \Delta_{\kk}^{x,d}$, and hence we recognize $\Ak= \mathbf{A_{\mathbb{Z}}} \otimes \kk$. 
Thus, the universal coefficients theorem in homological degree $0$ gives that $H_{x,0}(\mathbf{A_{\mathbb{Z}}}) \otimes \kk \xrightarrow{\cong} H_{x,0}(\Ak)$, implying the claimed reduction. 

To simplify notation we will not write $_{\mathbb{Z}}$ for the rest of this proof since we will only treat the case $\kk=\mathbb{Z}$. 
Consider the cell-attachment filtration $\mathbf{fA} \in \Alg_{E_2}((\sMod_{\mathbb{Z}}^{\mathsf{H}})^{\mathbb{Z}_{\leq}})$, which by \cite[Section 6.2.1]{Ek} is given by 
$$\mathbf{fA}= \mathbf{E_2}(S^{(1,0),0,0} \sigma_0 \oplus S^{(1,1),0,0} \sigma_1) \cup_{\sigma_1^2-\sigma_0^2}^{E_2}{D^{(2,0),1,1} \rho},$$
where the last grading represents the filtration stage. 
By \cite[Corollary 10.17]{Ek} there is a spectral sequence 
$$E^1_{x,p,q}=H_{x,p+q,p}(\overline{\mathbf{E_2}(S^{(1,0),0,0} \sigma_0 \oplus S^{(1,1),0,0} \sigma_1 \oplus S^{(2,0),1,1} \rho)}) \Rightarrow H_{x,p+q}(\mathbf{\overline{A}}).$$
The first page of this spectral sequence can be accessed by \cite[Theorems 16.4, 16.7]{Ek} and the description of the homology operation $Q_{\mathbb{Z}}^1(-)$ in \cite[Page 199]{Ek}. 
In homological degrees $p+q \leq 1$ the full answer is given by
\begin{enumerate}[$\bullet$]
    \item $E^1_{x,p,-p}$ vanishes for $p \neq 0$, and $\bigoplus_{x \in \mathsf{H}}{E^1_{x,0,0}}$ is the free $\mathbb{Z}$-module on the set of generators $\{\sigma_0^a \cdot \sigma_1^b: \; a+b \geq 0\}$, where $\sigma_0^0=\sigma_1^0=1$. 
    \item The only elements in homological degree $p+q=1$ are stabilizations by powers of $\sigma_0$ and $\sigma_1$ of one of the classes $\rho$, $Q_{\mathbb{Z}}^1(\sigma_0)$, $Q_{\mathbb{Z}}^1(\sigma_1)$, or $[\sigma_0,\sigma_1]$. 
    Thus they have filtration $p \leq 1$. 
\end{enumerate}

By \cite[Section 16.6]{Ek} the spectral sequence is multiplicative and its differential satisfies $d^1(\sigma_0)=0$, $d^1(\sigma_1)=0$ and $d^1(\rho)=\sigma_1^2-\sigma_0^2$. 
Thus,
$\bigoplus_{x \in \mathsf{H}} E^2_{x,0,0}$ is given as a ring by 
$\mathbb{Z}[\sigma_0,\sigma_1]/(\sigma_1^2-\sigma_0^2)$.
Hence to finish the proof it suffices to show that $E^2_{x,0,0}=E^{\infty}_{x,0,0}$ for any $x \in \mathsf{H}$.
This holds because $d^r$ decreases filtration by $r$ and homological degree by $1$ so $d^r$ vanishes on all the elements of homological degree $1$ for $r \geq 2$.  
\end{proof}

\subsection{Proof of Theorem \ref{theorem stab 1}} \label{proof thm stab 1}
\begin{proof}
We will do a series of reductions to get to the case $\X=\Ak$ for some appropriate coefficients $\kk$, and then we will do a direct computation. 

\textbf{Step 1.} The aim of this step is to reduce to the particular case $\X=\Ak$. 

Each class $\sigma_{\epsilon} \in H_{(1,\epsilon),0}(\X)$ is represented by a homotopy class $\sigma_{\epsilon}: S_{\kk}^{(1,\epsilon),0} \rightarrow \X$. 
Thus there is an $E_2$-algebra map $\mathbf{E_2}(S_{\kk}^{(1,0),0} \sigma_0 \oplus S_{\kk}^{(1,1),0} \sigma_1) \rightarrow \X$ sending $\sigma_0,\sigma_1$ to the corresponding homology classes of $\X$.  
By assumption $\sigma_1^2-\sigma_0^2=0 \in H_{(2,0),0}(\X)$, so picking a nullhomotopy gives an extension to an $E_2$-algebra map $c: \Ak \rightarrow \X$. 

Now we claim that for the map $c$ we have $H_{x,d}^{E_2}(\X,\Ak)=0$ for $d<\rk(x)/2$. 
Indeed, by assumption $H_{x,d}^{E_2}(\X)=0$ for $d<\rk(x)-1$ and by Proposition \ref{prop Ak} $H_{x,d}^{E_2}(\Ak)=0$ for $d<\rk(x)-1$ too. 
Thus, it suffices to show the claim for $(\rk(x)=1,d=0)$. 
Both $\X$ and $\Ak$ are reduced, i.e. $H_{0,0}(-)$ vanishes on both of them, and hence by \cite[Corollary 11.12]{Ek} it suffices to show that $H_{x,0}(\X,\Ak)=0$ for $\rk(x)=1$, which holds by our assumption about the 0th homology of $\X$ and Proposition \ref{prop Ak}. 

Now let us suppose that the theorem holds for $\X=\Ak$. 
By \cite[Corollary 15.10]{Ek} with $\rho(x)= \rk(x)/2$, $\mu(x)=(\rk(x)-1)/2$ and $\mathbf{M}=\mathbf{\overline{A}_{\mathds{k}}}/\sigma_{\epsilon}$ 
we find that $H_{x,d}(B(\mathbf{\overline{X}}, \mathbf{\overline{A_{\kk}}},\mathbf{M}))=0$ for $d< \mu(x)$. 
But, by \cite[Section 12.2.4]{Ek}, $B(\mathbf{\overline{X}}, \mathbf{\overline{A_{\kk}}},\mathbf{M}) \simeq \mathbf{\overline{X}}/\sigma_{\epsilon}$, giving the required reduction. 
Note that we use the ``adapters construction'' of \cite[Section 12.2]{Ek} to view $\mathbf{M}$ as a left $\mathbf{\overline{A_{\kk}}}$-module. 
We will also use this construction in the rest of this proof without explicit mention. 

\textbf{Step 2.} Now we will further reduce to the case $\mathbf{A}_{\F}$, for $\ell$ a prime number or $0$, where $\mathbb{F}_0:=\mathbb{Q}$. 

Recall from the proof of Proposition \ref{prop Ak} that $\Ak= \mathbf{A_{\mathbb{Z}}} \otimes \kk$.
Since base-change preserves colimits, the cofibration sequence 
$S_{\kk}^{(1,\epsilon),0} \otimes \mathbf{\overline{A_{\kk}}} \xrightarrow{\sigma_{\epsilon} \cdot - } \mathbf{\overline{A_{\kk}}}\rightarrow \mathbf{\overline{A_{\kk}}}/\sigma_{\epsilon}$ shows that 
 $\mathbf{\overline{A_{\kk}}}/\sigma_{\epsilon} \cong \mathbf{\overline{A_{\mathbb{Z}}}}/\sigma_{\epsilon} \otimes \kk$. 
Thus, by the universal coefficients theorem it suffices to prove the case $\kk=\mathbb{Z}$. 

We claim that the homology groups of $\mathbf{\overline{A_{\mathbb{Z}}}}$ are finitely generated. 
Indeed, by \cite[Theorem 16.4]{Ek} each entry on the first page of the spectral sequence of the proof of Proposition \ref{prop Ak} is finitely generated. 
Observe that this is a quite general fact that holds for any cellular $E_2$-algebra with finitely many cells by considering the analogous cell attachment filtration. 
We will use this again in the proof of Theorem \ref{theorem stab 2}. 

Hence, by the homology long exact sequence of  $\sigma_{\epsilon} \cdot -$, the homology groups of $\mathbf{\overline{A_{\mathbb{Z}}}}/\sigma_{\epsilon}$ are also finitely generated. 

Thus, it suffices to check the cases $\kk=\F$ for $\ell$ a prime number or $0$, by another application of the universal coefficients theorem and using the finite generation of the homology groups. 

\textbf{Step 3.} We will prove the theorem for a fixed $\ell$, $\kk=\F$ and $\X= \Ak$. 
To simplify notation we will not write the subscripts $\F$ for the rest of this proof. 
Consider the cell attachment filtration $\mathbf{fA} \in \Alg_{E_2}((\sMod_{\F}^{\mathsf{H}})^{\mathbb{Z}_{\leq}})$ as in the proof of Proposition \ref{prop Ak}. 
The filtration $0$ part is given by 
$\mathbf{\overline{fA}}(0)=\overline{\mathbf{E_2}(S^{(1,0),0} \sigma_0 \oplus S^{(1,1),0} \sigma_1)}$ so we can lift the maps $\sigma_{\epsilon}$ to filtered maps 
$\sigma_{\epsilon}: S^{(1,\epsilon),0,0} \rightarrow \mathbf{\overline{fA}}$.
Thus, using adapters we can form the left $\mathbf{\overline{fA}}$-module $\mathbf{\overline{fA}}/\sigma_{\epsilon}$ filtering $\mathbf{\overline{A}}/\sigma_{\epsilon}$.

Since $\gr(-)$ commutes with pushouts and with $\overline{(-)}$ by \cite[Lemma 12.7]{Ek}, we get two spectral sequences

\begin{enumerate}[(i)]
\item $\Scale[0.9]{F^1_{x,p,q}=H_{x,p+q,p}(\overline{\mathbf{E_2}(S^{(1,0),0,0} \sigma_0 \oplus S^{(1,1),0,0} \sigma_1 \oplus S^{(2,0),1,1} \rho)}) \Rightarrow H_{x,p+q}(\mathbf{\overline{A}})}$
\item $\Scale[0.9]{E^1_{x,p,q}=H_{x,p+q,p}(\overline{\mathbf{E_2}(S^{(1,0),0,0} \sigma_0 \oplus S^{(1,1),0,0} \sigma_1 \oplus S^{(2,0),1,1} \rho)}/\sigma_{\epsilon}) \Rightarrow H_{x,p+q}(\mathbf{\overline{A}}/\sigma_{\epsilon})}.$
\end{enumerate}

In order to prove the theorem it suffices to show the following claim 
\begin{claim}
$E^2_{x,p,q}=0$ for $p+q<(\rk(x)-1)/2$.
\end{claim}

We will need some preparation. 
As in the proof of Proposition \ref{prop Ak}, the first spectral sequence is multiplicative and its differential satisfies $d^1(\sigma_0)=0$, $d^1(\sigma_1)=0$ and $d^1(\rho)=\sigma_1^2-\sigma_0^2$. 
Moreover, by \cite[Theorem 16.4, Section 16.2]{Ek},
its first page is given by $\Lambda(L)$ where $\Lambda(-)$ denotes the free graded-commutative algebra, and $L$ is the $\F$-vector space with basis $Q_{\ell}^I(y)$ such that $y$ is a basic Lie word in $\{\sigma_0, \sigma_1, \rho\}$ and $I$ is admissible, in the sense of \cite[Section 16.2]{Ek}.

The second spectral sequence is a module over the first one, and we can identify $E^1=F^1/(\sigma_{\epsilon})$ because $\sigma_{\epsilon} \cdot -$ is injective in $F^1$, by the above description of $F^1$, and its image is the ideal $(\sigma_{\epsilon})$.
Therefore $E^1= \Lambda(L/\F\{\sigma_{\epsilon}\})$ and hence the $d^1$ differential in $F^1$ completely determines the $d^1$ differential in $E^1$, making it into a CDGA. 

\begin{proof}[Proof of Claim.]
$E^2$ is given by the homology of the CDGA $(E^1,d^1)$, and to prove the result we will introduce a ``computational filtration'' in this CDGA that has the virtue of filtering away most of the $d^1$ differential.

We let $\mathcal{F}^{\bullet} E^1$ be the filtration in which $\sigma_{1-\epsilon}$ and $\rho$ are given filtration $0$, the remaining elements of a homogeneous basis of $L/\F\{\sigma_{\epsilon}\}$ extending these are given filtration equal to their homological degree, and then we extend the filtration to $\Lambda(L/\F\{\sigma_{\epsilon}\})$ multiplicatively. 

Since $d^1$ preserves this filtration we get a spectral sequence converging to $E^2$ whose first page is the homology of the associated graded $\gr(\mathcal{F}^{\bullet} E^1)$. 
Thus, it suffices to show that $H_*(\gr(\mathcal{F}^{\bullet} E^1))$ already has the required vanishing line. 

Let us denote by $D$ the corresponding differential on this computational spectral sequence. 
Since $d^1$ lowers homological degree by $1$ we can decompose $(\gr(\mathcal{F}^{\bullet} E^1),D)$ as a tensor product 
$$(\Lambda(\F\{\sigma_{1-\epsilon},\rho\}),D) \otimes_{\F} (\Lambda(L/\F\{\sigma_0,\sigma_1,\rho\}),0),$$
where $D$ satisfies $D(\sigma_{1-\epsilon})=0$ and $D(\rho)=(-1)^{\epsilon} \sigma_{1-\epsilon}^2$. 
By the Künneth theorem the homology of this tensor product is $\F\{1,\sigma_{1-\epsilon}\} \otimes_{\F} \Lambda(L/\F\{\sigma_0,\sigma_1,\rho\})$ when $l \neq 2$, because graded-commutativity forces $\rho^2=0$; and  $\mathbb{F}_2\{1,\sigma_{1-\epsilon}\} \otimes_{\mathbb{F}_2} \mathbb{F}_2[\rho^2] \otimes_{\mathbb{F}_2} \Lambda(L/\mathbb{F}_2\{\sigma_0,\sigma_1,\rho\})$ if $l=2$.

By the \textit{slope} of an element we shall mean the ratio between its homological degree and the rank of its $\mathsf{H}$-valued grading. 
Since $\rho^2$ has slope $1/2$ and $\sigma_{1-\epsilon}$ has homological degree $0$ and rank $1$, in order to prove the required vanishing line it suffices to show that all the elements in $L/\F\{\sigma_0,\sigma_1,\rho\}$ have slope $\geq 1/2$. 
Since the slope of $Q_{\ell}^I(y)$ is always larger than or equal to the one of $y$, and the slope of the Browder bracket of two elements is always greater than the minimum of their slopes, the only elements in $L$ that have slope less than $1/2$ are those in the span of $\sigma_0, \sigma_1$, giving the result. 
\end{proof} \end{proof}

\subsection{Proof of Theorem \ref{theorem stab 2}}
\begin{proof}
The idea of the proof is identical to the previous one, so we will not spell out all the details, but we will focus instead in the extra complications that arise in the computations, specially in the later steps.  

\textbf{Step 1.} We will construct a certain cellular $E_2$-algebra $\mathbf{S_{\kk}}$ and show that it suffices to prove that $H_{x,d}(\mathbf{\overline{S_{\kk}}}/\sigma_{\epsilon})=0$ for $3d \leq 2 \rk(x)-4$. 

The assumptions of the statement imply that $[\sigma_0,\sigma_1]= \sigma_{\epsilon} \cdot y$ for some $y \in H_{(1,1-\epsilon),1}(\X)$, that $Q_{\kk}^1(\sigma_1)= \sigma_{\epsilon} \cdot x + t Q_{\kk}^1(\sigma_0)$ for some $x \in H_{(1,\epsilon),1}(\X)$ and some $t \in \mathbb{Z}$, and that $\sigma_{1-\epsilon} \cdot Q_{\kk}^1(\sigma_0)= \sigma_{\epsilon}^2 \cdot z  \in H_{(3,1-\epsilon),1}(\X)$ for some $z \in H_{(1,1-\epsilon),1}(\X)$.

Let 
\begin{equation*}
    \begin{aligned}
        \mathbf{S_{\kk}}:=\mathbf{E_2}(S_{\kk}^{(1,0),0} \sigma_0 \oplus S_{\kk}^{(1,1),0} \sigma_1 \oplus S_{\kk}^{(1,\epsilon),1} x \oplus S_{\kk}^{(1,1-\epsilon),1} y \oplus S_{\kk}^{(1,1-\epsilon),1}z) \\\cup_{\sigma_1^2-\sigma_0^2}^{E_2}{D_{\kk}^{(2,0),1} \rho}  \cup_{Q_{\kk}^1(\sigma_1)-\sigma_{\epsilon} \cdot x-t Q_{\kk}^1(\sigma_0)}^{E_2} {D_{\kk}^{(2,0),2} X} \cup_{[\sigma_0,\sigma_1]-\sigma_{\epsilon} \cdot y}^{E_2}{D_{\kk}^{(2,1),2} Y} \\ \cup_{\sigma_{1-\epsilon} \cdot Q_{\kk}^1(\sigma_0)- \sigma_{\epsilon}^2 \cdot z}^{E_2}{D_{\kk}^{(3,1-\epsilon),2} Z} \in \Alg_{E_2}(\sMod_{\kk}^{\mathsf{H}})
    \end{aligned}
\end{equation*}

By proceeding as in Step 1 of the proof of Theorem \ref{theorem stab 1} there is an $E_2$-algebra map $f: \mathbf{S_{\kk}} \rightarrow \X$ sending each of $\sigma_0,\sigma_1, x,y,z$ to the corresponding homology classes in $\X$ with the same name. 

\begin{claim}
$H_{x,d}^{E_2}(\X,\mathbf{S_{\kk}})=0$ for $d<2\rk(x)/3$.     
\end{claim}

Assuming the claim, we can apply \cite[Corollary 15.10]{Ek} with $\rho(x)= 2\rk(x)/3$, $\mu(x)=(2\rk(x)-3)/3$ and $\mathbf{M}=\mathbf{\overline{S_{\kk}}}/\sigma_{\epsilon}$ to obtain the required reduction. 
Thus, to finish this step we just need to show the claim. 

\begin{proof}[Proof of Claim.]
Proceeding as in the proof of Proposition \ref{prop Ak} one can compute $Q_{\mathbb{L}}^{E_2}(\mathbf{S_{\kk}})$ and check that $H_{x,d}^{E_2}(\mathbf{S_{\kk}})=0$ for $d<\rk(x)-1$. 
Since $\X$ has the same vanishing line on its $E_2$-homology it suffices to check that $H_{x,d}^{E_2}(\X,\mathbf{S_{\kk}})=0$ for $(\rk(x)=1, d=0)$ and $(\rk(x)=2, d=1)$. 

For $(\rk(x)=1,d=0)$ we use \cite[Corollary 11.12]{Ek} to reduce it to showing that $H_{x,0}(\X,\mathbf{S_{\kk}})=0$, as in Step 1 in the proof of Theorem \ref{theorem stab 1}. 
This holds because the 0th homology of $\X$ in rank $1$ is generated by $\sigma_0,\sigma_1$, which factor trough $f$ by construction. 

For $(\rk(x)=2,d=1)$ the argument will be more elaborate but use the same ideas. 
Pick sets of $\kk$-module generators $\{u_a\}_{a \in A}$ for $H_{0,1}(\X)$ and $\{v_b\}_{b \in B}$ for $H_{(1,0),1}(\X) \oplus H_{(1,1),1}(\X)$, where each $v_b$ has $\mathsf{H}$-grading $(1,\epsilon_b)$ for some $\epsilon_b \in \{0,1\}$. 
Consider $\mathbf{\Tilde{S}_{\kk}}:= \mathbf{S_{\kk}} \oplus^{E_2} \mathbf{E_2}(\bigoplus_{a \in A} S_{\kk}^{0,1} u_a \oplus \bigoplus_{b \in B} S_{\kk}^{(1,\epsilon_b),1} v_b)$.

The map $f: \mathbf{S_{\kk}} \rightarrow \X$ factors trough the canonical map $\mathbf{S_{\kk}} \rightarrow \mathbf{\Tilde{S}_{\kk}}$ in the obvious way, so we get a long exact sequence in $E_2$-homology for the triple $\mathbf{S_{\kk}} \rightarrow \mathbf{\Tilde{S}_{\kk}} \rightarrow \X$: 
$$ \cdots \rightarrow H_{x,1}^{E_2}(\mathbf{\Tilde{S}_{\kk}},\mathbf{S_{\kk}}) \rightarrow H_{x,1}^{E_2}(\X,\mathbf{S_{\kk}}) \rightarrow H_{x,1}^{E_2}(\X,\mathbf{\Tilde{S}_{\kk}}) \rightarrow \cdots.$$
The first term vanishes by direct computation of $Q_{\mathbb{L}}^{E_2}(\mathbf{\Tilde{S}_{\kk}})$ because $\rk(x)=2$, so it suffices to show that the third term vanishes too. 
Using \cite[Corollary 11.12]{Ek} it suffices to show that $H_{x',d'}(\X,\mathbf{\Tilde{S}_{\kk}})=0$ for $d' \leq 1$ and $\rk(x') \leq 2$. 

For a given $x' \in \mathsf{H}$ with $\rk(x') \leq 2$ we have an exact sequence 
\[\Scale[0.86]{\cdots \rightarrow H_{x',1}(\mathbf{\Tilde{S}_{\kk}}) \rightarrow H_{x',1}(\X) \rightarrow H_{x',1}(\X,\mathbf{\Tilde{S}_{\kk}}) \rightarrow H_{x',0}(\mathbf{\Tilde{S}_{\kk}}) \rightarrow H_{x',0}(\X) \rightarrow H_{x',0}(\X,\mathbf{\Tilde{S}_{\kk}}) \rightarrow 0},\]

so it suffices to show that $H_{x',1}(\mathbf{\Tilde{S}_{\kk}}) \rightarrow H_{x',1}(\X)$ is surjective and that $H_{x',0}(\mathbf{\Tilde{S}_{\kk}}) \rightarrow H_{x',0}(\X)$ is an isomorphism. 

The isomorphism in degree $0$ holds because $H_{*,0}(\mathbf{\Tilde{S}_{\kk}})=\kk[\sigma_0,\sigma_1]/(\sigma_1^2-\sigma_0^2)$ as a ring:
the proof is analogous to the computation of the 0th homology of $\Ak$ in the proof of Proposition \ref{prop Ak} because the extra cells that we have in $\mathbf{\Tilde{S}_{\kk}}$ are either in degree $\geq 2$ or in degree $1$ but attached trivially, so they have no effect in the homological degree $0$ part of the spectral sequence. 

The surjectivity in degree $1$ holds by construction if $\rk(x') \leq 1$. 
When $\rk(x')=2$ it holds by assumptions (i) and (ii) in the statement of the theorem plus the surjectivity in ranks $\leq 1$. 
\end{proof}

\textbf{Step 2.}
Now we will further reduce it to the case $\mathbf{S_{\FF}}$ for $\ell$ an odd prime or $0$. 

By proceeding as in Proposition \ref{prop Ak} we find that $\mathbf{S_{\kk}}= \mathbf{S}_{\mathbb{Z}[1/2]} \otimes_{\mathbb{Z}[1/2]} \kk$, so it suffices to consider the case $\kk=\mathbb{Z}[1/2]$ by the universal coefficients theorem. 

By reasoning as in Step 2 in the proof of Theorem \ref{theorem stab 1} we get that the homology groups of $\mathbf{S}_{\mathbb{Z}[1/2]}/ \sigma_{\epsilon}$ are finitely generated ${\mathbb{Z}[1/2]}$-modules because $\mathbf{S}_{\mathbb{Z}[1/2]}$ only has finitely many $E_2$-cells. 
Thus, another application of the universal coefficients theorem allows us to reduce to the case $\kk=\F$ with $\ell$ either an odd prime or $0$. 

\textbf{Step 3.}
Since we are working with $\F$-coefficients for a fixed $\ell$ we will drop the $\ell$ and $\F$ subscripts from now on. 
Let us begin by considering the cellular attachment filtration of $\mathbf{S}$, see \cite[Section 6.2.1]{Ek} for details, where the last grading denotes the filtration.
\begin{equation*}
    \begin{aligned}
        \mathbf{fS}:=\mathbf{E_2}(S^{(1,0),0,0} \sigma_0 \oplus S^{(1,1),0,0} \sigma_1 \oplus S^{(1,\epsilon),1,0} x \oplus S^{(1,1-\epsilon),1,0} y \oplus S^{(1,1-\epsilon),1,0} z) \\ \cup_{\sigma_1^2-\sigma_0^2}^{E_2}{D^{(2,0),1,1} \rho} \cup_{Q^1(\sigma_1)-\sigma_{\epsilon} \cdot x-t Q^1(\sigma_0)}^{E_2} {D^{(2,0),2,1} X} \cup_{[\sigma_0,\sigma_1]-\sigma_{\epsilon} \cdot y}^{E_2}{D^{(2,1),2,1} Y} \\ \cup_{\sigma_{1-\epsilon} \cdot Q^1(\sigma_0)- \sigma_{\epsilon}^2 \cdot z}^{E_2}{D^{(3,1-\epsilon),2,1} Z} \in \Alg_{E_2}((\sMod_{\F}^{\mathsf{H}})^{\mathbb{Z}_{\leq}})
    \end{aligned}
\end{equation*}

This gives two spectral sequences as in Step 3 of the proof of Theorem \ref{theorem stab 1}:
\begin{enumerate}[(i)]
    \item $F^1_{x,p,q}=H_{x,p+q,p}(\overline{\gr(\mathbf{fS})}) \Rightarrow H_{x,p+q}(\mathbf{\overline{S}})$
    \item $E^1_{x,p,q}=H_{x,p+q,p}(\overline{\gr(\mathbf{fS})}/\sigma_{\epsilon}) \Rightarrow H_{x,p+q}(\mathbf{\overline{S}}/\sigma_{\epsilon}).$
\end{enumerate}

The first spectral sequence is multiplicative, its first page is $\Lambda(L)$ where $L$ is the $\F$-vector space with basis $Q^I(u)$ such that $u$ a basic Lie word in $\{\sigma_0,\sigma_1,x,y,z,\rho,X,Y,Z\}$ and $I$ is admissible; and its $d^1$-differential satisfies $d^1(\sigma_0)=0$, $d^1(\sigma_1)=0$, $d^1(x)=0$, $d^1(y)=0$, $d^1(z)=0$, $d^1(\rho)=\sigma_1^2-\sigma_0^2$, $d^1(X)=Q^1(\sigma_1)-\sigma_{\epsilon} \cdot x-t Q^1(\sigma_0)$, $d^1(Y)=[\sigma_0,\sigma_1]-\sigma_{\epsilon} \cdot y$ and $d^1(Z)=\sigma_{1-\epsilon} \cdot Q^1(\sigma_0)-\sigma_{\epsilon}^2 \cdot z$. 

The second spectral sequence has the structure of a module over the first one, and its first page is $E^1= \Lambda(L/\F\{\sigma_{\epsilon}\})$, so $(E^1,d^1)$  has the structure of a CDGA. 

Thus, in order to finish the proof it suffices to show that $E^2_{x,p,q}=0$ for $p+q<(2 \rk(x)-3)/3$.

We will show the required vanishing line on $E^2$ by introducing a filtration on the CDGA $(E^1,d^1)$, similar to the one in Step 3 of the proof of Theorem \ref{theorem stab 1}. 
We let $\mathcal{F}^{\bullet} E^1$ be the filtration in which $\sigma_{1-\epsilon}$, $x$, $y$, $z$, $\rho$, $Q^1(\sigma_0)$, $Q^1(\sigma_1)$, $[\sigma_0,\sigma_1]$, $X$, $Y$, $Z$ are given filtration $0$, the remaining elements of a homogeneous basis of $L/\F\{\sigma_{\epsilon}\}$ extending these are given filtration equal to their homological degree, and we extend the filtration to $\Lambda(L/\F\{\sigma_{\epsilon}\})$ multiplicatively. 

This gives a spectral sequence converging to $E^2$ whose first page is the homology of the associated graded of the filtration $\mathcal{F}^{\bullet} E^1$. 
We will show the vanishing line on the first page of this spectral sequence. 

Applying \cite[Theorems 16.7 and 16.8]{Ek} gives that $d^1([\sigma_0,\sigma_1])=0$, $d^1(Q^1(\sigma_0))=0$ and $d^1(Q^1(\sigma_1))=0$. 
This allows to split the associated graded as a tensor product

\begin{equation*}
    \begin{aligned}
    (\gr(\mathcal{F}^{\bullet} E^1),D)= (\Lambda(\F\{\sigma_{1-\epsilon},\rho,Q^1(\sigma_0),Z,Q^1(\sigma_1),X\}),D)  \otimes_{\F} \\ (\Lambda(\F\{[\sigma_0,\sigma_1],Y\}),D)  \otimes_{\F} (\Lambda(\F\{x,y,z\}),0) \otimes_{\F} \\ (\Lambda(L/\F\{\sigma_0,\sigma_1,x,y,z,\rho,Q^1(\sigma_0),Q^1(\sigma_1),[\sigma_0,\sigma_1],X,Y,Z\}),0)
    \end{aligned}
\end{equation*}
where $D$ is the induced differential and satisfies $D(\sigma_{1-\epsilon})=0$, $D(\rho)=(-1)^{\epsilon} \sigma_{1-\epsilon}^2$, $D(Q^1(\sigma_0))=0$, $D(Z)=\sigma_{1-\epsilon} \cdot Q^1(\sigma_0)$, $D(Q^1(\sigma_1))=0$, $D(X)=Q^1(\sigma_1)-t Q^1(\sigma_0)$, $D([\sigma_0,\sigma_1])=0$, $D(Y)=[\sigma_0,\sigma_1]$. 
By the Künneth theorem it suffices to compute the homology of each of the factors separately. 

By direct computation we see that
\begin{enumerate}[$\bullet$]
    \item Elements in $\Lambda(L/\F\{\sigma_0,\sigma_1,x,y,z,\rho,Q^1(\sigma_0),Q^1(\sigma_1),[\sigma_0,\sigma_1],X,Y,Z\})$ have slope $\geq 2/3$.
    \item Elements in $\Lambda(\F\{x,y,z\})$ have slope $\geq 1 \geq 2/3$. 
    \item Since $\ell \neq 2$ we have $[\sigma_0,\sigma_1]^2=0$ so the homology of $(\Lambda(\F\{[\sigma_0,\sigma_1],Y\}),D)$ is $\F[Y^{\ell}]+ [\sigma_0,\sigma_1] \cdot \F\{Y^j: \ell | j+1\}$. 
    Since $[\sigma_0,\sigma_1]$ has bidegree $(\rk=2,d=1)$, $Y$ has bidegree $(\rk=2,d=2)$ and $\ell \geq 3$ then all these elements have slope $\geq 5/6 \geq 2/3$. 
\end{enumerate}
Thus, it suffices to check that $H_*(\Lambda(\F\{\sigma_{1-\epsilon},\rho,Q^1(\sigma_0),Z,Q^1(\sigma_1),X\}),D)$ vanishes for $3 d<2 \rk-3$, where $d$ denotes the homological degree. 
The remaining of the proof will be about studying this CDGA. 
We will separate this as an extra step because it will require some additional filtrations and work.

\textbf{Step 4.}
We firstly claim that it suffices to consider $t=0$: 
$\sigma_{1-\epsilon}$, $\rho$, $Q^1(\sigma_0)$, $Q^1(\sigma_1)$, $X$, $Z$ are now just the generators of a certain CDGA. 
Since both $Q^1(\sigma_0)$ and $Q^1(\sigma_1)$ lie in $\ker(D)$ and have the same homological degree and rank, the change of variables $Q^1(\sigma_1) \mapsto Q^1(\sigma_1)-t Q^1(\sigma_0)$ reparameterises $t \mapsto 0$. 

Secondly, once we are in the case $t=0$, we can further split the CDGA as a tensor product 
$$(\Lambda(\F\{\sigma_{1-\epsilon},\rho,Q^1(\sigma_0),Z\}),D) \otimes_{\F} (\Lambda(X,Q^1(\sigma_1)),D)$$
and the homology of the second factor is $\F[X^{\ell}]+ Q^1(\sigma_1) \cdot \F\{X^j: \ell | j+1\}$ (since $\ell \neq 2$), so all its elements have slope $\geq 5/6 \geq 2/3$. 
Thus, it suffices to prove that $H_*(\Lambda(\F\{\sigma_{1-\epsilon},\rho,Q^1(\sigma_0),Z\}),D)$ vanishes for $3 d<2 \rk-3$. 
For this, we will introduce an additional filtration by giving $Q^1(\sigma_0)$ filtration $0$, $\sigma_{1-\epsilon}$ filtration $1$ and $\rho, Z$ filtration $2$, and then extending the filtration multiplicatively to the whole CDGA. 

The differential $D$ preserves this filtration and the associated graded splits as a tensor product 
$$(\Lambda(\sigma_{1-\epsilon},\rho),D(\sigma_{1-\epsilon})=0, D(\rho)=(-1)^{\epsilon} \sigma_{1-\epsilon}^2) \otimes_{\F} (\Lambda(Q^1(\sigma_0),Z),0)$$
so, using that $\ell \neq 2$ to compute the homology of the first factor, we get a multiplicative spectral sequence of the form 
$$\mathcal{E}^1= \F[\sigma_{1-\epsilon}]/(\sigma_{1-\epsilon}^2) \otimes_{\F} \Lambda(Q^1(\sigma_0),Z) \Rightarrow H_{*}(\Lambda(\F\{\sigma_{1-\epsilon},\rho,Q^1(\sigma_0),Z\}),D)$$
whose first differential satisfies $D^1(Z)=\sigma_{1-\epsilon} \cdot Q^1(\sigma_0)$, $D^1(\sigma_{1-\epsilon})=0$ and $D^1(Q^1(\sigma_0))=0$. 

To finish the proof we will establish the required vanishing range on $\mathcal{E}^2$. 
To do so, we write $\mathcal{E}^1=\F\{1,\sigma_{1-\epsilon},Q^1(\sigma_0),\sigma_{1-\epsilon} \cdot Q^1(\sigma_0)\} \otimes \F[Z]$ as a $\F$-vector space, and then compute $\ker(D^1),\im(D^1)$ explicitly as $\F$-vector spaces, where $()$ denotes the ideal generated by an element:
$$\ker(D^1)=(\sigma_{1-\epsilon})+(Q^1(\sigma_0))+\F[Z^{\ell}]$$
and
$$\im(D^1)=\sigma_{1-\epsilon} \cdot Q^1(\sigma_0) \cdot \F\{Z^i: \ell \nmid i+1\}.$$
Thus, we get that $\mathcal{E}^2=\ker(D^1)/\im(D^1)$ is, as a $\F$-vector space, given by 
$$\mathcal{E}^2=\F[Z^l]+\sigma_{1-\epsilon} \cdot \F[Z]+ Q^1(\sigma_0) \cdot \F[Z]+ \sigma_{1-\epsilon} \cdot Q^1(\sigma_0) \cdot \F\{Z^i: \ell \mid i+1\}.$$
Using the bidegrees of the generators we find that the first summand vanishes for $d<2\rk/3$, the second vanishes for $d<2(\rk-1)/3$, the third one for $d<(2\rk-1)/3$, and the forth one for $d <(2 \rk-3)/3$, as required. 
\end{proof}

\subsection{Construction of the class $\theta$} \label{construction theta}
In this section we will explain how the class $\theta \in H_{(4,0),2}(\X)$ of Theorem \ref{thm stab 3} is defined. 

The first step will be to define $\theta \in H_{(4,0),2}(\AFF)$. 
Since we will only work with $\FF$-coefficients for now, we will drop all the $\FF$-indices.
Consider the spectral sequence (i) of the proof of Theorem \ref{theorem stab 1}: 
$$F^1_{x,p,q}=H_{x,p+q,p}(\overline{\mathbf{E_2}(S^{(1,0),0,0} \sigma_0 \oplus S^{(1,1),0,0} \sigma_1 \oplus S^{(2,0),1,1} \rho)}) \Rightarrow H_{x,p+q}(\mathbf{\overline{A}}).$$
As we said, this is a multiplicative spectral sequence whose first page is given by $\FF[L]$, where $L$ is the $\FF$-vector space with basis $Q^I(y)$ such that $y$ is a basic Lie word in $\{\sigma_0,\sigma_1,\rho\}$ and $I$ is admissible. 
(Note that this time we get a free commutative algebra instead of graded-commutative as we work with $\FF$-coefficients.)
Thus we have $F^1_{(4,0),2,0}=\FF\{\rho^2\}$.

\begin{claim}
$\rho^2$ survives to $F^{\infty}$. 
\end{claim}

\begin{proof}
Since $F^1_{(4,0),2+r,1-r}=0$ for $r \geq 1$ then $\rho^2$ cannot be a boundary of any $d^r$-differential. 
Moreover, $d^r: F^r_{(4,0),2,0} \rightarrow F^r_{(4,0),2-r,r-1}$ vanishes for $r>2$ since $\mathbf{fA}$ vanishes on negative filtration. 
Thus, it suffices to show that both $d^1(\rho^2)$ and $d^2(\rho^2)$ vanish. 
By Leibniz rule we have $d^1(\rho^2)=0$, so we only need to show that $d^2(\rho^2)=0$. 

Since $\rho^2=Q^1(\rho)$ and $d^1(\rho)=\sigma_1^2-\sigma_0^2$ then \cite[Theorem 16.8 (i)]{Ek} gives that $d^2(\rho^2)$ is represented by $Q^1(\sigma_1^2-\sigma_0^2)$.
(As a technical note let us mention that the result we just quoted is stated for $E_{\infty}$-algebras, but the same result holds for $E_2$-algebras as explained in \cite[Page 184]{Ek}.)
Finally, $Q^1(\sigma_1^2-\sigma_0^2)$ vanishes by the properties of $Q^1$ shown in \cite[Section 16.2.2]{Ek}. 
\end{proof}

\begin{definition}\label{def theta}
The class $\theta \in H_{(4,0),2}(\AFF)$ is defined to be any lift of the class $[\rho^2] \in F^{\infty}_{(4,0),2,0}$.

Given $\X$ satisfying the assumptions of Theorem \ref{thm stab 3} we define $\theta \in H_{(4,0),2}(\X)$ as follows: we pick an $E_2$-map $\mathbf{A} \xrightarrow{c} \mathbf{X}$ as in Step 1 in the proof of Theorem \ref{theorem stab 1}, and set $\theta:= c_*(\theta) \in H_{(4,0),2}(\mathbf{X})$. 
\end{definition}

\begin{rem} \label{rem indeterminacy}
There is not a unique choice of class $\theta$, however the statement of Theorem \ref{thm stab 3} will be true for \textit{any choice} of class $\theta$ with the property of Definition \ref{def theta}. 
In fact, $\theta$ is well-defined up to adding any linear combination of $Q^1(\sigma_0)^2$, $Q^1(\sigma_0) \cdot Q^1(\sigma_1)$, $Q^1(\sigma_1)^2$ and $[\sigma_0,\sigma_1]^2$, or multiples of $\sigma_0^2=\sigma_1^2$ to it. 
(This fact will not be needed in the rest of the paper but we added an explanation below.)
\end{rem}

In order to show the above remark one can use the same spectral sequence and check that 
\begin{enumerate}
    \item $F^1_{(4,0),0,2}$ is generated by $Q^1(\sigma_0)^2$, $Q^1(\sigma_0) \cdot Q^1(\sigma_1)$, $Q^1(\sigma_1)^2$ and $[\sigma_0,\sigma_1]^2$, and all these terms are permanent cycles and no boundaries. 
    \item $d^1: F^1_{(4,0),1,1} \rightarrow F^1_{(4,0),0,1}$ is injective, and hence $F^2_{(4,0),1,1}=0$. 
\end{enumerate}
Thus, $\theta \in H_{(4,0),2}(\AFF)$ is well-defined up to a linear combination of $Q^1(\sigma_0)^2$, $Q^1(\sigma_0) \cdot Q^1(\sigma_1)$, $Q^1(\sigma_1)^2$ and $[\sigma_0,\sigma_1]^2$. 

The definition of the map $c$ is not unique as we need to choose a nullhomotopy of $\sigma_1^2-\sigma_0^2$ in $\X$, and the set of such choices is a $H_{(2,0),1}(\X)$-torsor.
In particular, by assumptions (i) and (ii) about $\X$ any new choice of $\rho$ differs by a class in $\im(\sigma_{\epsilon} \cdot -)$ or by a multiple of $Q^1(\sigma_0)$, giving the result. 
\subsection{The proof of Theorem \ref{thm stab 3}} \label{section proof of thm stab 3}

Before proving the theorem let us briefly recall the construction of $\X/(\sigma_{\epsilon},\theta)$. 
We start by viewing $\theta$ as a homotopy class of maps $S^{(4,0),2} \rightarrow \mathbf{X}$. 
Then, using the adapters construction, see \cite[Section 12.3]{Ek} we get an $\mathbf{\overline{X}}$-module map $S^{(4,0),2} \otimes \mathbf{\overline{X}}/\sigma_{\epsilon} \xrightarrow{\theta \cdot-} \mathbf{\overline{X}}/\sigma_{\epsilon}$ and we define $\mathbf{\overline{X}}/(\sigma_{\epsilon},\theta)$ to be its cofibre (in the category of left $\mathbf{\overline{X}}$-modules). 

\begin{proof}
The proof will be very similar to that of Theorem \ref{theorem stab 2}, so we will focus on the parts that are different and skip details. 

\textbf{Step 1.}
We will construct a certain cellular $E_2$-algebra $\mathbf{S}$ and show that it suffices to prove that $H_{x,d}(\mathbf{\overline{S}}/(\sigma_{\epsilon},\theta))=0$ for $3d < 2 \rk(x)-4$. 

The assumptions of the statement imply that $[\sigma_0,\sigma_1]= \sigma_{\epsilon} \cdot y$ for some $y \in H_{(1,1-\epsilon),1}(\X)$, that $Q^1(\sigma_1)= \sigma_{\epsilon} \cdot x + t Q^1(\sigma_0)$ for some $x \in H_{(1,\epsilon),1}(\X)$ and some $t \in \FF$, and that $\sigma_{1-\epsilon} \cdot Q^1(\sigma_0)= \sigma_{\epsilon}^2 \cdot z  \in H_{(3,1-\epsilon),1}(\X)$ for some $z \in H_{(1,1-\epsilon),1}(\X)$.

Moreover, we claim that there is $u \in H_{(4,0),3}(\X)$ such that $Q^1(\sigma_0)^3=\sigma_{\epsilon}^2 \cdot u$.

Indeed, condition (iv) says that $\sigma_0 \cdot Q^1(\sigma_0)= \sigma_{\epsilon}^2 \cdot \tau$ for some $\tau \in H_{(1,0),1}(\X)$, and then we can apply $Q^2(-)$ to both sides and use the formulae in \cite[Section 16.2.2]{Ek} to find $Q^1(\sigma_0)^3+ \sigma_0^2 \cdot Q^2(Q^1(\sigma_0))+\sigma_0[\sigma_0,Q^1(\sigma_0)] Q^1(\sigma_0)= \sigma_{\epsilon}^2 \cdot [\sigma_{\epsilon},\sigma_{\epsilon}] \cdot Q^1(\tau)+\sigma_{\epsilon}^4 \cdot Q^2(\tau)+\sigma_{\epsilon}^2\cdot [\sigma_{\epsilon}^2,\tau] \cdot \tau$, hence the result as $\sigma_{\epsilon}^2=\sigma_0^2$ and as $[\sigma_0,Q^1(\sigma_0)]=[\sigma_0,[\sigma_0,\sigma_0]]=0$ (by \cite[Section 16.2.2]{Ek} again).

Let 
\begin{equation*}
    \begin{aligned}
        \mathbf{S}:=\mathbf{E_2}(S^{(1,0),0} \sigma_0 \oplus S^{(1,1),0} \sigma_1 \oplus S^{(1,\epsilon),1} x \oplus S^{(1,1-\epsilon),1} y \oplus S^{(1,1-\epsilon),1}z \oplus S^{(4,0),3}u) \\\cup_{\sigma_1^2-\sigma_0^2}^{E_2}{D^{(2,0),1} \rho}  \cup_{Q^1(\sigma_1)-\sigma_{\epsilon} \cdot x-t Q^1(\sigma_0)}^{E_2} {D^{(2,0),2} X} \cup_{[\sigma_0,\sigma_1]-\sigma_{\epsilon} \cdot y}^{E_2}{D^{(2,1),2} Y} \\ \cup_{\sigma_{1-\epsilon} \cdot Q^1(\sigma_0)- \sigma_{\epsilon}^2 \cdot z}^{E_2}{D^{(3,1-\epsilon),2} Z} \cup_{Q^1(\sigma_0)^3-\sigma_{\epsilon}^2 \cdot u}^{E_2}{D^{(6,0),4} U} \in \Alg_{E_2}(\sMod_{\FF}^{\mathsf{H}})
    \end{aligned}
\end{equation*}

By proceeding as in Step 1 of the proof of Theorem \ref{theorem stab 1} there is an $E_2$-algebra map $f: \mathbf{S} \rightarrow \X$ sending each of $\sigma_0,\sigma_1, x,y,z,u$ to the corresponding homology classes in $\X$ with the same name. 
Moreover, we can assume that $f$ extends any given map $\mathbf{A} \rightarrow \X$ and hence that it sends $\theta \mapsto \theta$. 

\begin{claim}
$H_{x,d}^{E_2}(\X,\mathbf{S})=0$ for $d<2 \rk(x)/3$.     
\end{claim}

The proof is identical to the corresponding claim in Step 1 in the proof of Theorem \ref{theorem stab 2}. 
The only difference now is that $\mathbf{S}$ has a cell $U$ below the ``critical line'' $d=\rk-1$. 
However, it causes no trouble since it has bidegree $(\rk=6,d=4)$, so it lies on the region $3d \geq 2 \rk$. 

Assuming the claim, we can apply \cite[Corollary 15.10]{Ek} with $\rho(x)= 2\rk(x)/3$, $\mu(x)=(2\rk(x)-4)/3$ and $\mathbf{M}=\mathbf{\overline{S}}/(\sigma_{\epsilon},\theta)$ to obtain the required reduction. 

\textbf{Setp 2.}
We proceed as in Step 3 in the proof of Theorem \ref{theorem stab 2} to get a cell attachment filtration $\mathbf{fS} \in \Alg_{E_2}((\sMod_{\FF}^{\mathsf{H}})^{\mathbb{Z}_{\leq}})$. 
The key now is to observe that $\theta \in H_{(4,0),2}(\mathbf{S})$ lifts to a filtered map $\theta: S^{(4,0),2,2} \rightarrow \mathbf{fS}$ which maps to $\rho^2 \in H_{*,*,*}(\gr(\mathbf{fS}))$. 

Indeed, $\theta \in H_{(4,0),2}(\mathbf{A})=H_{(4,0),2}(\colim(\mathbf{fA}))=\colim_f(H_{(4,0),2,f}(\mathbf{fA}))$, so it can be represented by a class $\theta \in H_{(4,0),2,f}(\mathbf{fA})$ for some $f$ large. 
In fact, $f=2$ is  the smallest possible such value since the obstruction to lift the class $\theta \in H_{(4,0),2,f}(\mathbf{fA})$ to a class in $ H_{(4,0),2,f-1}(\mathbf{fA})$ is precisely the image of $\theta$ in $H_{(4,0),2,f}(\gr(\mathbf{fA}))$ which is $\rho^2$ by definition, giving the result. 
Finally observe that there is a canonical map of filtered $E_2$-algebras $\mathbf{fA} \rightarrow \mathbf{fS}$. 
Thus, we get spectral sequences
\begin{enumerate}[(i)]
    \item $F^1_{x,p,q}=H_{x,p+q,p}(\overline{\gr(\mathbf{fS})}) \Rightarrow H_{x,p+q}(\mathbf{\overline{S}})$
    \item $E^1_{x,p,q}=H_{x,p+q,p}(\overline{\gr(\mathbf{fS})}/(\sigma_{\epsilon},\rho^2)) \Rightarrow H_{x,p+q}(\mathbf{\overline{S}}/(\sigma_{\epsilon},\theta)).$
\end{enumerate}

The first spectral sequence is multiplicative, its first page is $\FF[L]$ where $L$ is the $\FF$-vector space with basis $Q^I(\alpha)$ such that $\alpha$ a basic Lie word in $\{\sigma_0,\sigma_1,x,y,z,u,\rho,X,Y,Z,U\}$ and $I$ is admissible; and its $d^1$-differential satisfies $d^1(\sigma_0)=0$, $d^1(\sigma_1)=0$, $d^1(x)=0$, $d^1(y)=0$, $d^1(z)=0$, $d^1(u)=0$, $d^1(\rho)=\sigma_1^2-\sigma_0^2$, $d^1(X)=Q^1(\sigma_1)-\sigma_{\epsilon} \cdot x-t Q^1(\sigma_0)$, $d^1(Y)=[\sigma_0,\sigma_1]-\sigma_{\epsilon} \cdot y$, $d^1(Z)=\sigma_{1-\epsilon} \cdot Q^1(\sigma_0)-\sigma_{\epsilon}^2 \cdot z$ and $d^1(U)=Q^1(\sigma_0)^3-\sigma_{\epsilon}^2 \cdot u$. 

The second spectral sequence has the structure of a module over the first one, and its first page is given by $E^1= \FF[L/\FF\{\sigma_{\epsilon}\}]/(\rho^2)$ because $\sigma_{\epsilon} \cdot -$ is injective on $\FF[L]$ and $\rho^2 \cdot -$ is injective on $\FF[L]/(\sigma_{\epsilon})=\FF[L/\FF\{\sigma_{\epsilon}\}]$. 
Thus $(E^1,d^1)$  has the structure of a CDGA. 

Thus, in order to finish the proof it suffices to show that $E^2_{x,p,q}=0$ for $p+q<(2 \rk(x)-4)/3$.

\textbf{Step 3.}
Now we will introduce additional filtrations to simplify the CDGA until we get the required result. 
The first filtration is similar to the one of Step 3 in the proof of Theorem \ref{theorem stab 2}: we give $\sigma_{1-\epsilon}$, $x$, $y$, $z$, $u$, $\rho$, $Q^1(\sigma_0)$, $Q^1(\sigma_1)$, $[\sigma_0,\sigma_1]$, $X$, $Y$, $Z$, $U$ filtration $0$, we give the remaining elements of a homogeneous basis of $L/\FF\{\sigma_{\epsilon}\}$ extending these filtration equal to their homological degree, and we extend the filtration to $\FF(L/\FF\{\sigma_{\epsilon}\})/(\rho^2)$ multiplicatively (which we can as $\rho$ is in filtration $0$). 

This allows us to split the associated graded as a tensor product and all the factors are concentrated in the region $3d \geq 2 \rk$ except possibly the one given by 
$$(\FF[\sigma_{1-\epsilon},\rho,Q^1(\sigma_0),Q^1(\sigma_1),X,Z,U]/(\rho^2),D)$$
where the non-zero part of $D$ is characterized by $D(X)=Q^1(\sigma_1)-tQ^1(\sigma_0)$, $D(Z)=\sigma_{1-\epsilon} \cdot Q^1(\sigma_0)$ and $D(U)=Q^1(\sigma_0)^3$. 
(This computation is easier than the one of the proof of Theorem \ref{theorem stab 2} since $\ell=2$ simplifies the homology of the other factors.)

Then, we can proceed as in Step 4 in the proof of Theorem \ref{theorem stab 2} to go to the case $t=0$ and hence split the CDGA further to simplify it to 
$$(\FF[\sigma_{1-\epsilon},\rho,Q^1(\sigma_0),Z,U]/(\rho^2),D).$$
Next we introduce a new filtration by giving $\sigma_{1-\epsilon}, \rho, Q^1(\sigma_0)$ filtration $0$, and $Z,U$ filtration $1$ and then extending multiplicatively. 
The associated graded of this splits as a tensor product
$$(\FF[\sigma_{1-\epsilon},\rho]/(\rho^2), D(\rho)=\sigma_{1-\epsilon}^2) \otimes_{\FF} (\FF[Q^1(\sigma_0),Z,U],0)$$
and the homology of the first factor is precisely $\FF[\sigma_{1-\epsilon}]/(\sigma_{1-\epsilon}^2)$, yielding a spectral sequence of the form 
\[\Scale[0.9]{\mathcal{E}^1= \FF[\sigma_{1-\epsilon}]/(\sigma_{1-\epsilon}^2) \otimes_{\FF} \FF[Q^1(\sigma_0),Z,U] \Rightarrow H_*(\FF[\sigma_{1-\epsilon},\rho,Q^1(\sigma_0),Z,U]/(\rho^2),D)}\]
whose first differential $D^1$ satisfies $D^1(Z)= \sigma_{1-\epsilon} \cdot Q^1(\sigma_0)$ and $D^1(U)=Q^1(\sigma_0)^3$. 
We will establish the required vanishing line on $\mathcal{E}^2$ of this spectral sequence. 
For that we will introduce yet another filtration by letting $\sigma_{1-\epsilon}, Q^1(\sigma_0), U$ have filtration $0$ and $Z$ have filtration $1$. 

The associated graded is given by
$$(\FF[\sigma_{1-\epsilon},Z]/(\sigma_{1-\epsilon}^2),0) \otimes_{\FF} (\FF[Q^1(\sigma_0),U],\delta(U)=Q^1(\sigma_0)^3)$$
where $\delta$ is the new differential. 
Thus, its homology is given by 
$$\FF[\sigma_{1-\epsilon}, Q^1(\sigma_0), Z,U^2]/(\sigma_{1-\epsilon}^2,Q^1(\sigma_0)^3)$$
and the $\delta^1$-differential satisfies $\delta^1(Z)= \sigma_{1-\epsilon} \cdot Q^1(\sigma_0)$. 
Since $U$ has slope $2/3$ itself, in order to prove the required vanishing line we can just focus on the remaining part 
$$(\FF[\sigma_{1-\epsilon}, Q^1(\sigma_0), Z]/(\sigma_{1-\epsilon}^2,Q^1(\sigma_0)^3),\delta^1(Z)=\sigma_{1-\epsilon} \cdot Q^1(\sigma_0)).$$
For that we explicitly compute $\ker(\delta^1)$, $\im(\delta^1)$ as $\FF$-vector spaces (similar to the last CDGA of the proof of Theorem \ref{theorem stab 2}). 
$$\ker(\delta^1)= (\sigma_{1-\epsilon})+(Q^1(\sigma_0)^2)+\FF\{1,Q^1(\sigma_0)\} \cdot \FF[Z^2]$$
and
$$\im(\delta^1)=\FF\{\sigma_{1-\epsilon} \cdot Q^1(\sigma_0), \sigma_{1-\epsilon} \cdot Q^1(\sigma_0)^2\} \cdot \FF[Z^2].$$
Thus we get 
\begin{equation*}
    \begin{aligned}
        \ker(\delta^1)/\im(\delta^1)= \sigma_{1-\epsilon} \cdot \FF[Z]+ \sigma_{1-\epsilon} \cdot Q^1(\sigma_0) \cdot \FF\{Z^i: 2 \nmid i\}+ \\
        \sigma_{1-\epsilon} \cdot Q^1(\sigma_0)^2 \cdot \FF\{Z^i: 2 \nmid i\}+\FF\{1,Q^1(\sigma_0),Q(\sigma_0)^2\} \cdot \FF[Z].
    \end{aligned}
\end{equation*}

Using the bidegrees of the generators it is immediate that all vanish for $3d<2 \rk -4$, hence the result. 
\end{proof}

Finally we will finish the Section by giving the Corollary of of Theorem \ref{thm stab 3} which is used in Section \ref{section intro}. 

\begin{corollary} \label{cor 2 torsion}
Let $\X$ be as in Theorem \ref{thm stab 3} then 
\begin{enumerate}[(i)]
    \item If $\theta^3 \in H_{(12,0),6}(\X)$ does not destabilize by $\sigma_{\epsilon}$ then $H_{(4k,0),2k}(\mathbf{\overline{X}}/\sigma_{\epsilon}) \neq 0$ for all $k \geq 1$, and in particular the optimal slope for the stability is $1/2$. 
    \item If $\theta^3 \in H_{(12,0),6}(\X)$ destabilizes by either $\sigma_0$ or $\sigma_1$ then $H_{x,d}(\mathbf{\overline{X}}/\sigma_{\epsilon})=0$ for $3d \leq 2\rk(x)-6$, so $\X$ satisfies homological stability of slope at least $2/3$ with respect to $\sigma_{\epsilon}$. 
\end{enumerate}
\end{corollary}

\begin{proof}
By definition (using the adapters construction as in Section \ref{section proof of thm stab 3}) there is a cofibration of left $\mathbf{\overline{X}}$-modules
$$S^{(4,0),2} \otimes \mathbf{\overline{X}}/\sigma_{\epsilon} \xrightarrow{\theta \cdot -} \mathbf{\overline{X}}/\sigma_{\epsilon} \rightarrow \mathbf{\overline{X}}/(\sigma_{\epsilon},\theta),$$
and hence a corresponding long exact sequence in homology groups which implies that 
$\theta \cdot -: H_{x-(4,0),d-2}(\mathbf{\overline{X}}/\sigma_{\epsilon}) \rightarrow H_{x,d}(\mathbf{\overline{X}}/\sigma_{\epsilon})$ is surjective for $3d \leq 2 \rk(x)-5$
and an isomorphism for $3d \leq 2 \rk(x)-8$. 

Similarly, the cofibration of left $\mathbf{\overline{X}}$-modules
$S^{(1,\epsilon),0} \otimes \mathbf{\overline{X}} \xrightarrow{\sigma_{\epsilon} \cdot -} \mathbf{\overline{X}}\rightarrow \mathbf{\overline{X}}/\sigma_{\epsilon}$ gives another long exact sequence in homology groups.

\textbf{Proof of (i).} 
If $\theta^3$ does not destabilise by $\sigma_{\epsilon}$ then the second long exact sequence gives $\theta^3 \neq 0 \in H_{(12,0),6}(\mathbf{\overline{X}}/\sigma_{\epsilon})$. 
But 
$$\theta \cdot -: H_{4(k-1),2(k-1)}(\mathbf{\overline{X}}/\sigma_{\epsilon}) \rightarrow H_{4k,2k}(\mathbf{\overline{X}}/\sigma_{\epsilon})$$ 
is an isomorphism for $k \geq 4$, so $\theta^k \neq 0 \in H_{4k,2k}(\mathbf{\overline{X}}/\sigma_{\epsilon})$ for $k \geq 4$ (hence for $k \geq 1$).  

\textbf{Proof of (ii).}
If $3d \leq 2 \rk(x)-6$ then $3(d+2k) \leq 2 (\rk(x)+4k)-8$ for any $k \geq 1$ and hence the map $\theta^k \cdot-: H_{x,d}(\mathbf{\overline{X}}/\sigma_{\epsilon}) \rightarrow H_{x+(4k,0),d+2k}(\mathbf{\overline{X}}/\sigma_{\epsilon})$ is an isomorphism for any $k \geq 1$. 
Thus, it suffices to find some $k$ for which $\theta^k \cdot -$ is the $0$ map. 
We will show that in fact $k=6$ works. 

Since $\theta^3$ destabilises by either $\sigma_0$ or $\sigma_1$ then (using that $\sigma_0^2=\sigma_1^2$) $\theta^6= \alpha \cdot \sigma_{\epsilon}^2$ for some $\alpha \in H_{(22,0),12}(\mathbf{X})$. 
Thus, $\theta^6 \cdot -= \alpha \cdot (\sigma_{\epsilon}^2 \cdot -)$ as (homotopy classes of) maps $S^{(24,0),12} \otimes \mathbf{\overline{X}}/\sigma_{\epsilon} \rightarrow \mathbf{\overline{X}}/\sigma_{\epsilon}$. 
Thus, it suffices to show that $\sigma_{\epsilon}^2 \cdot -: S^{(2,0),0} \otimes \mathbf{\overline{X}}/\sigma_{\epsilon} \rightarrow \mathbf{\overline{X}}/\sigma_{\epsilon}$ is nullhomotopic. 
This is a special case of the following general fact, see \cite[Proposition 2.3]{maxime}, that if $\X$ is an object in a stable $\infty$-category ($\sMod_{\FF}^{\mathsf{H}}$ in our case) and $f: \X \rightarrow \X$ is a self-map then (any) induced morphism $\overline{f}: \X/f \rightarrow \X/f$ on the cofibre satisfies that $\overline{f}^2$ is nullhomotopic. 
\end{proof}

\begin{rem} \label{rem theta well-defined}
By Remark \ref{rem indeterminacy} we know that $\theta$ itself is not well-defined. 
However, the map $\theta \cdot -: H_{x-(4,0),d-2}(\mathbf{\overline{X}}/\sigma_{\epsilon}) \rightarrow H_{x,d}(\mathbf{\overline{X}}/\sigma_{\epsilon})$ is well-defined up to adding $Q^1(\sigma_0)^2 \cdot -$ and $Q^1(\sigma_0) \cdot Q^1(\sigma_1) \cdot -$, and the map $\theta^2 \cdot -$ is well-defined. 

This can shown by using Remark \ref{rem indeterminacy} and the assumptions on $\X$ about the classes $Q^1(\sigma_0)$, $Q^1(\sigma_1)$, $[\sigma_0,\sigma_1]$ plus the fact that $Q^1(\sigma_0)^3$ destabilises twice as explained in Step 1 of the proof of Theorem \ref{thm stab 3}, and using \cite[Proposition 2.3]{maxime} again. 
\end{rem}

\section{$E_2$-algebras from quadratic data} \label{section 4}

There is a general framework of how to get an $E_2$-algebra from a braided monoidal groupoid, see \cite[Section 17.1]{Ek}. 
In this section we will consider braided monoidal groupoids with the extra data of a strong braided monoidal functor to $\mathsf{Set}$, and we will observe that the ``Grothendieck construction'' yields another braided groupoid, called the ``\textit{associated quadratic groupoid''}, and hence another $E_2$-algebra. 

This construction will generalize the way quadratic symplectic groups are constructed from symplectic groups and the way that spin mapping class groups are related to mapping class groups, if we let the extra data be the set of quadratic refinements (hence the use of the term ``quadratic'').  

We will also study the relationship between the $E_2$-algebra of the original groupoid and the one of the associated quadratic groupoid; in particular Theorem \ref{theorem splitting complexes} and Corollary \ref{cor std connectivity} allow us to get some vanishing lines in the $E_2$-homology of the associated quadratic groupoid from knowledge of the original groupoid. 

\subsection{Definition and construction of the $E_2$-algebras} \label{section quadratic data}

Let us start by introducing some notation based on the one in \cite[Section 17]{Ek}.  
All the categories for the rest of this section are discrete. 
A \textit{braided monoidal groupoid} is a triple $(\G,\oplus,\mathds{1})$, where $\G$ is a groupoid, $\oplus$ a braided monoidal structure on $\G$ and $\mathds{1}$ the monoidal unit. 
For an object $x \in \G$ we write $\G_x:=\G(x,x)=\Aut_{\G}(x)$.
We can view any (discrete) monoid as an example of a monoidal groupoid where the only morphisms are the identity; for example $\mathbb{N}$ is naturally a symmetric monoidal groupoid, so in particular braided. 

\begin{definition} \label{defn quadratic data}
A quadratic data consists of a triple $(\G,\rk,Q)$ where
\begin{enumerate}[(i)]
\item $\G=(\G,\oplus,\mathds{1})$ is a braided monoidal groupoid such that $\G_{\mathds{1}}$ is trivial and for any objects $x,y \in \G$ the map $- \oplus -: \G_x \times \G_y \rightarrow \G_{x \oplus y}$ is injective.  
\item $\rk: \G \rightarrow \mathbb{N}$ is a braided monoidal functor such that 
$\rk^{-1}(0)$ consists precisely of those objects isomorphic to $\mathds{1}$. 
\item $Q: \G^{\mathsf{op}} \rightarrow \mathsf{Set}$ is a strong braided monoidal functor. 
\end{enumerate}
\end{definition}

Parts (i) and (ii) are precisely the assumptions needed to apply all the constructions of \cite[Section 17]{Ek}, and part (iii) is the extra ``quadratic'' data. 
One should think of $Q(x)$ as the set of ``quadratic refinements'' of the object $x$; and strong monoidality implies in particular that $Q(\mathds{1})$ is a one element set. 

\begin{definition}
Given a quadratic data $(\G,\rk,Q)$, its associated quadratic groupoid is the braided monoidal groupoid $(\Gq,\oplus^{\mathsf{q}},\mathds{1}^{\mathsf{q}})$ given by the Grothendieck construction $\G \wreath Q$, i.e. 
\begin{enumerate}[(i)]
    \item The set of objects of $\Gq$ is $\bigsqcup_{x \in \G}{Q(x)}$. 
    \item The sets of morphisms are given as follows: for $q \in Q(x)$ and $q' \in Q(x')$, $\Gq(q,q')=\{\phi \in \G(x,x'): Q(\phi)(q')=q\}$. 
    \item The braided monoidal structure $\oplus^{\mathsf{q}}$ is induced by the strong braided monoidality of $Q$, and the monoidal unit $\mathds{1}^{\mathsf{q}}$ is given by the only element in $Q(\mathds{1})$. 
\end{enumerate}
\end{definition}

Let us denote by $\rk^{\mathsf{q}}: \Gq \rightarrow \mathbb{N}$ the braided monoidal functor given by $q \in Q(x) \mapsto \rk(x)$. 
By construction the group $\Gq_{\mathds{1}^{\mathsf{q}}}$ is trivial and for any objects $q,q' \in \Gq$ the map $- \oplus^{\mathsf{q}} -: \Gq_q \times \Gq_{q'} \rightarrow \Gq_{q \oplus^{\mathsf{q}} q'}$ is injective. 
Also, $(\rk^{\mathsf{q}})^{-1}(0)$ consists precisely of those objects isomorphic to $\mathds{1}^{\mathsf{q}}$. 
Thus, $(\Gq,\oplus^{\mathsf{q}},\mathds{1}^{\mathsf{q}},\rk^{\mathsf{q}})$ satisfies all the assumptions of \cite[Section 17]{Ek}, so by \cite[Section 17.1]{Ek} there is $\Rq \in \Alg_{E_2}(\sSet^{\mathbb{N}})$ such that 
$$\Rq(n) \simeq \left\{ \begin{array}{lcc}
             \emptyset & if & n=0 \\
             \underset{[q] \in \pi_0(\Gq): \; \rk^{\mathsf{q}}(q)=n}{\bigsqcup}{B \Gq_q} & if & n>0.
             \end{array}
   \right.$$
We shall call $\Rq$ the \textit{quadratic $E_2$-algebra} associated to a quadratic data.    
Alternatively, in the explicit construction of $\Rq$ in \cite[Section 17.1]{Ek} we can perform the Kan extension along the projection $\Gq \rightarrow \pi_0(\Gq)$ instead of along $\Gq \xrightarrow{\rk^{\mathsf{q}}} \mathbb{N}$, and hence we can view $\Rq \in \Alg_{E_2}(\sSet^{\pi_0(\Gq)})$ such that 
$$\Rq([q]) \simeq \left\{ \begin{array}{lcc}
             \emptyset & if \; q \cong \mathds{1}^{\mathsf{q}} \\
             B \Gq_q & otherwise.
             \end{array}
   \right.$$

We will not distinguish between these two, as sometimes it will be more convenient to think of $\Rq$ as being $\mathbb{N}$-graded an other times as $\pi_0(\Gq)$-graded. 

\begin{rem} \label{remark path components}
When we view $\Rq$ as $\pi_0(\Gq)$-graded we have that $\Rq([q])$ is path-connected for any $[q] \neq 0 \in \pi_0(\Gq)$.
Thus, the strictly associative algebra $\overline{\Rq}$ satisfies that $\pi_0(\overline{\Rq}) \cong \pi_0(\mathsf{G^q})$ as monoids, where the monoid structure on the left-hand-side is induced by the product. 
In particular, the ring $H_{*,0}(\overline{\Rq})$ is determined by the monoidal structure of $\pi_0(\mathsf{G^q})$. 

\end{rem}

Similarly, we can apply the construction of \cite[Section 17.1]{Ek} to $(\G,\oplus,\mathds{1},\rk)$ to get 
$\R \in \Alg_{E_2}(\sSet^{\mathbb{N}})$ such that 
$$\R(n) \simeq \left\{ \begin{array}{lcc}
             \emptyset & if & n=0 \\
             \underset{[x] \in \pi_0(\G): \; \rk(x)=n}{\bigsqcup}{B \G_x} & if & n>0.
             \end{array}
   \right.$$   
We will refer to $\R$ as the \textit{non-quadratic $E_2$-algebra}.    
The obvious braided monoidal functor $\Gq \rightarrow \G$ then induces an $E_2$-algebra map $\Rq \rightarrow \R$.    

\subsection{$E_1$-splitting complexes of quadratic groupoids}

Recall \cite[Definition 17.9]{Ek} that given a monoidal groupoid $\G$ with a rank functor $\rk: \G \rightarrow \mathbb{N}$ satisfying properties (i) and (ii) of Definition \ref{defn quadratic data} and an element $x \in \G$, the \textit{$E_1$-splitting complex} $S^{E_1,\G}_{\bullet}(x)$ is the semisimplicial set with $p$-simplices given by 
$$S^{E_1,\G}_p(x):= \underset{(x_0,\cdots,x_{p+1}) \in \G_{\rk>0}^{p+2}}{\colim}{\G(x_0 \oplus \cdots \oplus x_{p+1},x)}$$
and face maps given by the monoidal structure. 
(Where $\G_{\rk>0}$ denotes the full subgroupoid of $\G$ on those objects $x$ with $\rk(x)>0$, i.e. on those objects not isomorphic to $\mathds{1}$.)

The main result of this section is the following result which will allow us to understand splitting complexes of quadratic groupoids. 

\begin{theorem} \label{theorem splitting complexes}
Let $(\G,\rk,Q)$ be a quadratic data, then for any $q \in Q(x)$ there is an isomorphism of semisimplicial sets $S_{\bullet}^{E_1,\Gq}(q) \cong S_{\bullet}^{E_1,\G}(x)$.    
\end{theorem}

\begin{proof}
By definition
$$S_{p}^{E_1,\G}(x)= \underset{(x_0,\cdots,x_{p+1}) \in \G_{\rk>0}^{p+2}}{\colim}{\G(x_0 \oplus \cdots \oplus x_{p+1},x)}$$
and 
$$S_p^{E_1,\Gq}(q)= \underset{(q_0,\cdots,q_{p+1}) \in \Gq_{\rk^{\mathsf{q}}>0}^{p+2}}{\colim}{\Gq(q_0 \oplus^{\mathsf{q}} \cdots \oplus^{\mathsf{q}} q_{p+1},q)}$$

The inclusions $\Gq(q_0 \oplus^{\mathsf{q}} \cdots \oplus^{\mathsf{q}} q_{p+1},q) \subset \G(x_0 \oplus \cdots \oplus x_{p+1},x)$, for each $(q_0,\cdots,q_{p+1}) \in \Gq_{\rk^{\mathsf{q}}>0}^{p+2}$ with $q_i \in Q(x_i)$ and $q \in Q(x)$, assemble into canonical maps
\[\Scale[0.75]{S_p^{E_1,\Gq}(q)=\underset{(q_0,\cdots,q_{p+1}) \in \Gq_{\rk^{\mathsf{q}}>0}^{p+2}}{\colim}{\Gq(q_0 \oplus^{\mathsf{q}} \cdots \oplus^{\mathsf{q}} q_{p+1},q)} \rightarrow \underset{(x_0,\cdots,x_{p+1}) \in \G_{\rk>0}^{p+2}}{\colim}{\G(x_0 \oplus \cdots \oplus x_{p+1},x)}=S_p^{E_1,\G}(x)}\]
which are compatible with the face maps of both semisimplicial sets because the natural functor $\Gq \rightarrow \G$ is monoidal. 
Thus, it suffices to show that $S_{p}^{E_1,\Gq}(q) \rightarrow S_{p}^{E_1,\G}(x)$ is a bijection of sets for all $p$. 

Surjectivity: 
any element on the right hand side is represented by some $\phi \in \G(x_0 \oplus \cdots \oplus x_{p+1},x)$ which is an isomorphism since $\G$ is a groupoid. 
Since $Q$ is strong monoidal, $Q(\phi): Q(x) \xrightarrow{\cong} Q(x_0) \times \cdots \times Q(x_{p+1})$ is an isomorphism.
Let $q_i:= \proj_i ( Q(\phi)(q)) \in Q(x_i)$, then $\phi \in \Gq(q_0 \oplus^{\mathsf{q}} \cdots \oplus^{\mathsf{q}} q_{p+1},q)$ defines an element on the left hand side mapping to the required element. 

Injectivity: 
suppose that two elements on the left hand side have the same image on the right hand side.
Represent them by $\phi^i \in \Gq(q_0^i \oplus^{\mathsf{q}} \cdots \oplus^{\mathsf{q}} q_{p+1}^i,q)$, where $Q(\phi^i)(q)=q_0^i \oplus^{\mathsf{q}} \cdots \oplus^{\mathsf{q}} q_{p+1}^i$ and $i \in \{1,2\}$ is an index. 

Since $\phi^1$ and $\phi^2$ agree on the colimit of the right hand side then there is an element $\phi \in \G(x_0 \oplus \cdots +x_{p+1},x)$ and morphisms $(\psi_0^i,\cdots ,\psi_{p+1}^i): (x_0^i,\cdots ,x_{p+1}^i) \rightarrow (x_0, \cdots ,x_{p+1})$ in $\G_{\rk>0}^{p+2}$ such that $\phi^i \circ (\psi_0^i, \cdots,\psi_{p+1}^i)^{-1}= \phi$ for $i \in \{1,2\}$. 

Let ${q'_a}^i:= Q({\psi_a^i}^{-1})(q_a^i) \in Q(x_a)$, we claim that ${q'_a}^1={q'_a}^2$ for $i \in \{1,2\}$: 
$Q(\psi_0^i \oplus \cdots \oplus \psi_{p+1}^i) ({q'_0}^i \oplus^{\mathsf{q}} \cdots \oplus^{\mathsf{q}} {q'_{p+1}}^i)= (q_0^i \oplus^{\mathsf{q}} \cdots \oplus^{\mathsf{q}} q_{p+1}^i)=Q(\phi^i)(q)$ and hence 
${q'_0}^i \oplus^{\mathsf{q}} \cdots \oplus^{\mathsf{q}} {q'_{p+1}}^i=Q(\phi^i\circ (\psi_0^i \oplus \cdots \oplus \psi_{p+1}^i)^{-1})(q)=Q(\phi)(q)$. 
Since $Q(\phi)(q)$ is independent of $i \in \{1,2\}$ then the claim follows by the strong monoidality $Q$ since ${q'_a}^1, {q'_a}^2 \in Q(x_a)$ for all $a$. 

Now let $q_a:={q'_a}^1={q'_a}^2$, then by definition $Q(\psi_0^i \oplus \cdots \oplus \psi_{p+1}^i) ({q}_0 \oplus^{\mathsf{q}} \cdots \oplus^{\mathsf{q}} {q}_{p+1})= (q_0^i \oplus^{\mathsf{q}} \cdots \oplus^{\mathsf{q}} q_{p+1}^i)$ and hence $(\psi_0^i, \cdots,\psi_{p+1}^i) \in \Gq_{\rk^{\mathsf{q}>0}}^{p+2}$. 
Since $\phi^i \circ (\psi_0^i, \cdots,\psi_{p+1}^i)^{-1}= \phi$ for $i \in \{1,2\}$ by construction, then $\phi^1$ and $\phi^2$ agree on the left hand side colimit, as required. 
\end{proof}

Recall \cite[Definition 17.6, Lemma 17.10]{Ek}: we say that $(\G,\oplus,\mathds{1},\rk)$ \textit{satisfies the standard connectivity estimate} if for any $x \in \G$, the reduced homology of $S^{E_1,\G}(x):=||S_{\bullet}^{E_1,\G}||$ is concentrated in degree $\rk(x)-2$. 
As explained in \cite[Page 188]{Ek} the standard connectivity estimate implies that $H_{n,d}^{E_1}(\R)=0$ for $d<n-1$, where $\R$ is the $E_2$-algebra defined in Section \ref{section quadratic data}.
The following corollary says that the standard connectivity estimate on the underlying braided groupoid of a quadratic data also gives a vanishing line on the $E_2$-homology of the corresponding quadratic $E_2$-algebra. 

\begin{corollary} \label{cor std connectivity}
If $(\G,\rk,Q)$ is a quadratic data such that $(\G,\rk)$ satisfies the standard connectivity estimate then $H_{x,d}^{E_2}(\Rq)=0$ for $d<\rk(x)-1$.   
\end{corollary}

\begin{proof}
    By Theorem \ref{theorem splitting complexes} and the standard connectivity estimate, the reduced homology of $S^{E_1,\Gq}(q)$ is concentrated in degree $\rk^{\mathsf{q}}(q)-2$ for any $q \in \Gq$. 
    Thus, by \cite[Page 188]{Ek} we have $H_{x,d}^{E_1}(\Rq)=0$ for $d<\rk(x)-1$.  
    Finally, the ``transferring vanishing lines up'' theorem, \cite[Theorem 14.4]{Ek} implies the result. 
\end{proof}

\section{Quadratic symplectic groups} \label{section symplectic}

\subsection{Construction of the $E_2$-algebra}
For a given $g \geq 0$ we let the \textit{standard symplectic form} on $\mathbb{Z}^{2g}$ be the matrix $\Omega_g$ given by the block diagonal sum of $g$ copies of $\begin{psmallmatrix}
    0 & 1 \\ -1 & 0
\end{psmallmatrix}.$ 
The \textit{genus $g$ symplectic group} is defined by $Sp_{2g}(\mathbb{Z}):= \Aut(\mathbb{Z}^{2g},\Omega_g)$. 

Let $(\Sp,\oplus,0)$ be the symmetric monoidal groupoid with objects the non-negative integers, morphisms 
$\Sp(g,h)=\left\{ \begin{array}{lcc}
             Sp_{2g}(\mathbb{Z}) & if \; g=h \\
            \emptyset & otherwise,
             \end{array}
   \right.$
where the (strict) monoidal structure $\oplus$ is given by addition on objects and block diagonal sum on morphisms, $0$ is the (strict) monoidal unit and the braiding $\beta_{g,h}: g \oplus h \xrightarrow{\cong} h \oplus g$ is given by the matrix 
$\begin{psmallmatrix} 0 & I_{2h} \\ I_{2g} & 0 \end{psmallmatrix}$, which satisfies $\beta_{h,g} \beta_{g,h}=\id_{g+h}$. 

We let $\rk: \Sp \rightarrow \mathbb{N}$ be the symmetric monoidal functor given by identity on objects
and let $Q: \mathsf{Sp^{op}} \rightarrow \mathsf{Set}$ be the functor given as follows
\begin{enumerate}[(i)]
    \item On objects, $Q(g):=\{q: \mathbb{Z}^{2g} \rightarrow \mathbb{Z}/2: \; q(x+y)\equiv q(x)+ q(y)+ x \cdot y (\mod 2) \}$, where $\cdot$ represents the skew-symmetric product induced by the standrd symplectic form. 
    \item On morphisms, for $\phi \in Sp_{2g}(\mathbb{Z})$ and $q \in Q(g)$ we let $Q(\phi)(q)= q \circ \phi$.
\end{enumerate}

In other words, $Q(g)$ is the set of quadratic refinements on $(\mathbb{Z}^{2g},\Omega_g)$, as defined in Section \ref{section intro}. 
Strong symmetric monoidality of $Q$ follows from the fact that a quadratic refinement $q \in Q(g)$ is the same data as a function of sets from a basis of $\mathbb{Z}^{2g}$ to $\mathbb{Z}/2$. 
Thus, $(\Sp,\rk,Q)$ is a quadratic data in the sense of Definition \ref{defn quadratic data}. 

By Section \ref{section quadratic data} we get an associated quadratic groupoid $\mathsf{Sp^q}$ and a quadratic $E_2$-algebra $\mathbf{R^{\mathsf{q}}}$, which in this case is actually $E_{\infty}$ because the groupoid is symmetric and not just braided; however, this will not make a difference for the purposes of this paper. 

The next goal is to describe $\pi_0(\mathsf{Sp^q})$, which by Remark \ref{remark path components} gives a computation of $H_{*,0}(\overline{\Rq})$. 
In order to do so, we need to introduce the so called \textit{Arf invariant}. 

\begin{definition} \label{defn arf}
Given a quadratic refinement $q \in Q(g)$ of the standard symplectic form on $\mathbb{Z}^{2g}$, we define the Arf invariant of $q$ via $\Arf(q):=\sum_{i=1}^{g}{q(e_i) q(f_i)} \in \mathbb{Z}/2$, where $(e_1,f_1,\cdots,e_g,f_g)$ is the standard (ordered) basis of $\mathbb{Z}^{2g}$.
\end{definition}

The key property of this invariant is that for $q,q' \in Q(g)$ we have $\Arf(q)=\Arf(q')$ if and only if there exists $\phi \in Sp_{2g}(\mathbb{Z})$ such that $q'=Q(\phi)(q)$. 
Moreover, for $g \geq 1$ it is clear that $\Arf: Q(g) \rightarrow \mathbb{Z}/2$ is surjective. 

Before stating the next result recall the monoid $\mathsf{H}:= \{0\} \cup \mathbb{N}_{>0} \times \mathbb{Z}/2$, where the monoidal structure $+$ is given by addition in both coordinates, considered at the beginning of Section \ref{section Ek}.

\begin{lemma} \label{lem arf inv}
Taking rank and Arf invariant gives an isomorphism of monoids $(\rk,\Arf): \pi_0(\mathsf{Sp^q}) \xrightarrow{\simeq} \mathsf{H}$. 
\end{lemma}

\begin{proof}
The map $(\rk,\Arf): \pi_0(\mathsf{Sp^q}) \rightarrow \mathsf{H}$ is clearly surjective, it is injective and well-defined by the above discussion of the Arf invariant, and it is monoidal because $\rk$ is monoidal and $\Arf$ is also monoidal by its explicit formula.  
\end{proof}

Under this identification of $\pi_0(\Rq)$ we have that $\Rq(g,\epsilon)=B Sp_{2g}^{\epsilon}(\mathbb{Z})$ is the classifying space of a quadratic symplectic group in the sense of Section \ref{section results}. 
Thus, by Section \ref{section E2 algebras overview}, Theorem \hyperref[theorem B]{B} is equivalent to certain vanishing lines of $H_{*,*}(\mathbf{\overline{\Rq}}/\sigma_{\epsilon};\mathds{k})$ and $H_{*,*}(\mathbf{\overline{\Rq}}/(\sigma_{\epsilon},\theta);\FF)$.
\
\subsection{Proof of Theorem \hyperref[theorem B]{B}}

The only additional ingredient that we need to prove Theorem \hyperref[theorem B]{B} is to understand the $E_1$-splitting complex of $(\Sp,\rk)$. 

\begin{proposition} \label{prop std connect}
$(\Sp,\rk)$ satisfies the standard connectivity estimate, i.e. for $g \in \mathbb{N}$ the reduced homology of $S^{E_1,\Sp}(g)$ is concentrated in degree $g-2$.     
\end{proposition}

\begin{proof}
Let $P(g)$ be the poset whose elements are submodules $0 \subsetneq M \subsetneq \mathbb{Z}^{2g}$ such that $(M,\Omega_g|_{M})$ is isomorphic to the standard symplectic form $(\mathbb{Z}^{2h},\Omega_h)$ for some $0<h<g$, ordered by inclusion. 
Let $P_{\bullet}(g)$ be the nerve of the poset, viewed as a semisimplicial set with $p$-simplices strict chains $M_0 \subsetneq M_1 \subsetneq \cdots \subsetneq M_p$ in $P(g)$, and face maps given by forgetting elements in the chain. 

The first step in the proof is about comparing $P_{\bullet}(g)$ with the $E_1$-splitting complex.

\begin{claim}
There is an isomorphism of semisimplicial sets $S_{\bullet}^{E_1,\Sp}(g) \rightarrow P_{\bullet}(g)$. 
\end{claim}

\begin{proof}
By \cite[Remark 17.11]{Ek} we have the following more concrete description of $S_{\bullet}^{E_1,\Sp}(g)$: 
$$S_{p}^{E_1,\Sp}(g)= \bigsqcup_{(g_0,\cdots,g_{p+1}): \; g_i>0, \; \sum_{i} g_i= g}{\frac{Sp_{2g}(\mathbb{Z})}{Sp_{2g_0}(\mathbb{Z}) \times Sp_{2g_1}(\mathbb{Z}) \times \cdots \times Sp_{2g_{p+1}}(\mathbb{Z})}}$$
with the obvious face maps. 

For each $0<n<g$ we let $M_n:=\mathbb{Z}^{2n} \oplus 0 \subset \mathbb{Z}^{2g}$, so that we have a chain $M_1< \cdots < M_{g-1}$ in $P(g)$. 
For each tuple $(g_0,\cdots,g_{p+1})$ with $g_i>0$ and $\sum_{i}{g_i}=g$ we have a $p$-simplex $\sigma_{g_0,\cdots,g_{p+1}}:= M_{g_0}< M_{g_0+g_1} < \cdots < M_{g_0+\cdots+g_p} \in P_p(g)$.

The group $Sp_{2g}(\mathbb{Z})$ acts simplicially on $P_{\bullet}(g)$, and under this action the stabilizer of $\sigma_{g_0,\cdots,g_{p+1}}$ is precisely $Sp_{2g_0}(\mathbb{Z}) \times Sp_{2g_1}(\mathbb{Z}) \times \cdots \times Sp_{2g_{p+1}}(\mathbb{Z}) \subset Sp_{2g}(\mathbb{Z})$.
Thus, we indeed get a levelwise injective map of semisimplicial sets $S_{\bullet}^{E_1,\Sp}(g) \rightarrow P_{\bullet}(g)$.

Levelwise surjectivity follows from the fact that for a given $M \in P(g)$, any isomorphism  $(M,\Omega_g|_M) \xrightarrow{\cong} (\mathbb{Z}^{\rk{M}},\Omega_{\rk(M)/2})$ can be extended to an automorphism of $(\mathbb{Z}^{2g},\Omega_g)$. 
This is a consequence of the classification of non-degenerate skew-symmetric forms over finitely generated free $\mathbb{Z}$-modules.
\end{proof}

Let us denote $L:= (\mathbb{Z}^{2g},\Omega_g)$. 
The poset $P(g)$ is then the same as $\mathcal{U}(L)_{0<-<L}$ in the sense of \cite[Section 1]{spherical}. 
By \cite[Theorem 1.1]{spherical} the poset $\mathcal{U}(L)$ is Cohen-Macaulay of dimension $g$, and in particular the poset $\mathcal{U}(L)_{0<-<L}$ is $(g-3)$-connected and $(g-2)$-dimensional, giving the result.
\end{proof}

\begin{proof}[Proof of Theorem]
\textbf{Part (i).} 
By Lemma \ref{lem arf inv} and Remark \ref{remark path components} we have $\Rq \in \Alg_{E_2}(\sSet^{\mathsf{H}})$ such that $\Rq(x)$ is path-connected for each $x \in \mathsf{H}\setminus \{0\}$ and $\Rq(0)=\emptyset$. 
Thus, $H_{0,0}(\Rq)=0$ and $H_{*,0}(\mathbf{\overline{\Rq}})=\mathbb{Z}[\sigma_0,\sigma_1]/(\sigma_1^2-\sigma_0^2)$ as a ring, where $\sigma_{\epsilon}$ is generated by a point in $\Rq((1,\epsilon))$. 
By Proposition \ref{prop std connect}, $(\Sp,\rk)$ satisfies the standard connectivity estimate, and thus by Corollary \ref{cor std connectivity} we get that $H_{x,d}^{E_2}(\Rq)=0$ for $d<\rk(x)-1$. 

If we now consider $\mathbf{X}:= \mathbf{\Rq_{\mathbb{Z}}} \in \Alg_{E_2}(\sMod_{\mathbb{Z}}^{\mathsf{H}})$ then it satisfies the assumptions of Theorem \ref{theorem stab 1} by \cite[Lemma 18.2]{Ek} and the properties of $(-)_{\mathbb{Z}}$ explained in Section \ref{section E2 algebras overview}. 
Thus the claimed homological stability for $\Rq$ follows. 

\textbf{Part (ii).}
This time let $\mathbf{X}:= \mathbf{\Rq_{\mathbb{Z}[1/2]}} \in \Alg_{E_2}(\sMod_{\mathbb{Z}[1/2]}^{\mathsf{H}})$. 
As before, this algebra satisfies the unnumbered assumptions of Theorem \ref{theorem stab 2}. 
We will check that it also verifies assumptions (i),(ii) and (iii) and then the required stability will follow from Theorem \ref{theorem stab 2}. 
By the universal coefficients theorem, to check them it suffices to prove that 
 $H_{x,1}(\Rq)$ is $2$-torsion for $\rk(x) \in \{2,3\}$, which follows from Theorems \ref{thm: 6.7} and \ref{thm: 6.8} in the Appendix.

\textbf{Secondary stability.}
Let $\mathbf{X}:= \mathbf{\Rq_{\FF}} \in \Alg_{E_2}(\sMod_{\FF}^{\mathsf{H}})$. 
Then Theorem \ref{thm stab 3} applies by Theorems \ref{thm: 6.7} and \ref{thm: 6.8} in the Appendix and the universal coefficients theorem. 
The result then follows by the long exact sequence of the cofibration 
$S^{(4,0),2} \otimes \mathbf{\overline{X}}/\sigma_{\epsilon} \rightarrow \mathbf{\overline{X}}/\sigma_{\epsilon} \rightarrow \mathbf{\overline{X}}/(\sigma_{\epsilon},\theta).$
\end{proof}

\section{Spin mapping class groups} \label{section mcg}

Consider the braided monoidal groupoid $(\MCG,\oplus,0)$ defined in \cite[section 4]{E2} whose objects are the non-negative integers and morphisms are given by 
$$\MCG(g,h)=\left\{ \begin{array}{lcc}
             \Gamma_{g,1} &if \; g=h \\
            \emptyset & otherwise.
             \end{array}
   \right.$$
The monoidal structure on $\MCG$ is given by addition on objects and by ``gluing diffeomorphisms'' on morphisms, using the decomposition of $\Sigma_{g+h,1}$ as a boundary connected sum $\Sigma_{g,1} \natural \Sigma_{h,1}$. 
The braiding is induced by the half right-handed Dehn twist along the boundary. 
Let $\rk: \MCG \rightarrow \mathbb{N}$ be the braided monoidal functor given by identity on objects. 

Let $Q: \mathsf{MCG^{op}} \rightarrow \mathsf{Set}$ be the functor given as follows
\begin{enumerate}[(i)]
    \item On objects, 
    $$Q(g)=\{q:H_1(\Sigma_{g,1};\mathbb{Z}) \rightarrow \mathbb{Z}/2: \; q(x+y) \equiv q(x)+q(y)+x \cdot y (\mod 2)\},$$
    where $\cdot$ is the homology intersection pairing. 
    \item On morphisms, for $\phi \in \Gamma_{g,1}$ and $q \in Q(g)$ we let $Q(\phi)(q)=q \circ \phi_*$. 
\end{enumerate}

In other words, $Q(g)$ is the set of quadratic refinements of the intersection product in $H_1(\Sigma_{g,1};\mathbb{Z})$, which is isomorphic to the standard hyperbolic form of genus $g$. 
By the argument of Section \ref{section symplectic}, $Q$ is strong braided monoidal, so $(\MCG,\rk,Q)$ is a quadratic data. 
Moreover, by mimicking the proof of Lemma \ref{lem arf inv} we get that $(\rk,\Arf): \pi_0(\mathsf{MCG^q}) \xrightarrow{\simeq} \mathsf{H}$ is a monoidal isomorphism. 
(This uses the surjectivity of the map $\Gamma_{g,1} \rightarrow Sp_{2g}(\mathbb{Z})$.)

\begin{rem}
Using \cite[Section 2]{rspin} one can check that $\Rq$ agrees with the ``moduli space of spin surfaces with one boundary component'', defined in more geometric terms using tangential structures. 
\end{rem}

Since the $E_2$-algebra $\Rq$ satisfies that $\Rq(n,\epsilon) \simeq B \Gamma_{g,1}^{1/2}[\epsilon]$, Theorem \hyperref[theorem A]{A} is equivalent to certain vanishing lines in the homology of $\Rq/\sigma_{\epsilon}$ and $\Rq/(\sigma_{\epsilon},\theta)$.  

\subsection{Proof of Theorem \hyperref[theorem A]{A}}

\begin{proof}
In this case the standard connectivity estimate for $(\MCG,\rk)$ is proven in \cite[Theorem 3.4]{E2}. 
Thus, proceeding as in the proof of Theorem \hyperref[theorem B]{B} we can apply Theorem \ref{theorem stab 1} to $\mathbf{\Rq_{\mathbb{Z}}}$ to get part (i) of the Theorem. 

To prove part (ii) we consider $\mathbf{X}:=\mathbf{\Rq_{\mathbb{Z}[1/2]}}$ and apply Theorem \ref{theorem stab 2}. 
To verify assumptions (i),(ii) and (iii) we use the universal coefficients theorem and Theorems \ref{thm: 6.1}, \ref{thm: 6.2}, \ref{thm: 6.3} and \ref{thm: 6.4} in the Appendix.  

The secondary stability part follows by considering $\mathbf{X}:=\mathbf{\Rq_{\FF}}$ and applying Theorem \ref{thm stab 3}, where all assumptions needed hold by Theorems \ref{thm: 6.1}, \ref{thm: 6.2}, \ref{thm: 6.3} and \ref{thm: 6.4} in the Appendix.  
\end{proof}

As we said in Section \ref{section intro} we can also prove that the bound of Theorem \hyperref[theorem A]{A} is (almost) optimal. 

\begin{lemma}\label{lem optimallity}
For all $k \geq 1$ and for all $\epsilon, \delta \in \{0,1\}$ the map
$$\sigma_{\epsilon} \cdot -: H_{2k}(B\Gamma_{3k-1,1}^{1/2}[\delta-\epsilon];\mathbb{Q}) \rightarrow H_{2k}(B\Gamma_{3k,1}^{1/2}[\delta];\mathbb{Q})$$
is not surjective.
\end{lemma}

\begin{proof}
Suppose for a contradiction that it was surjective for some $k \geq 1$, $\epsilon, \delta \in \{0,1\}$.
By the transfer, $H_{2k}(B\Gamma_{3k,1}^{1/2}[\delta];\mathbb{Q}) \rightarrow H_{2k}(B\Gamma_{3k,1};\mathbb{Q})$ is also surjective since the spin mapping class groups are fine index subgroups of the mapping class groups. 

Thus, the stabilisation map $\sigma \cdot -:H_{2k}(B\Gamma_{3k-1,1};\mathbb{Q}) \rightarrow H_{2k}(B\Gamma_{3k,1};\mathbb{Q})$ (where we are using the notation of \cite{E2}) must be surjective. 
By the universal coefficients theorem, $H^{2k}(B\Gamma_{3k,1};\mathbb{Q}) \rightarrow H^{2k}(B\Gamma_{3k-1,1};\mathbb{Q})$ is injective; which is false by the computations in \cite[Proof of Corollary 5.8]{E2}. 
\end{proof}

Thus, the stability bound obtained in Theorem \hyperref[theorem A]{A} is optimal up to an additive constant of at most $4/3$.

\section{Appendix} \label{appendix}

\subsection{Spin mapping class groups} \label{appendix mcg}

In this section we will explain the homology computations of spin mapping class groups and quadratic symplectic groups. 
These computations are done using GAP and we have included the code that we used. 

\subsubsection{$g=1$} \label{genus 1}

Consider the simple closed curves $\alpha, \beta$ in $\Sigma_{1,1}$ shown below, which orientations chosen so that $\alpha \cdot \beta = +1$.
Let $a, b \in \Gamma_{1,1}$ be the isotopy classes represented by the right-handed Dehn twists along the curves $\alpha$ and $\beta$ respectively. 

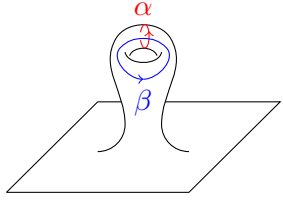
\begin{figure}[H]
    \begin{center}
        \begin{tikzpicture}[scale=1.2, decoration={
                markings,
                mark=at position 0.5 with {\arrow{<}}}
            ] 
		\draw (0,0) -- (1,1);
		\draw (1,1) -- (3,1) -- (2,0) -- (0,0);
		\churro[scale=1, x=1, y=0.45]
		\draw[red,dashed] (1.5,1.85) node[above] {$\alpha$}
		    to[in=120,out=-120] (1.5,1.6);
		\draw[red,postaction={decorate}] (1.5,1.85) to[in=-30,out=30] (1.5,1.6);
		\draw[blue,->] (1.5,1.25) node[below] {$\beta$}
		    to[in=-45,out=0] (1.75,1.6)
		    to[in=0,out=90+45] (1.5,1.7)
		    to[in=45,out=180] (1.25,1.6)
		    to[in=180,out=-90-45]  (1.5,1.25);
    \end{tikzpicture}
    \end{center}
\end{figure}

The set $Q(1)$ of quadratic refinements of $H_1(\Sigma_{1,1};\mathbb{Z})$ is $\{q_{0,0}, q_{1,0}, q_{0,1}, q_{1,1}\}$, where $q_{i,j}$ satisfies $q_{i,j}(a)=i$ and $q_{i,j}(b)=j$. 
The first three of them have Arf invariant $0$ and the forth one has Arf invariant $1$. 

Thus, we can get explicit models of $\Gamma_{1,1}^{1/2}[\epsilon]$ via $\Gamma_{1,1}^{1/2}[0]:=\Stab_{\Gamma_{1,1}}(q_{0,0})$ and $\Gamma_{1,1}^{1/2}[1]:=\Stab_{\Gamma_{1,1}}(q_{1,1})$. 

\begin{theorem}\label{thm: 6.1}
\begin{enumerate}[(i)]
    \item $H_1(\Gamma_{1,1};\mathbb{Z})=\mathbb{Z}\{\tau\}$, where $\tau$ is represented by both $a$ and $b$. 
    \item $H_1(\Gamma_{1,1}^{1/2}[0];\mathbb{Z})=\mathbb{Z}\{x\} \oplus \mathbb{Z}\{y\}$, where $x$ is represented by $a^{-2}$ and $y$ is represented by $a b a^{-1}$. 
    Moreover, $b^{-2}$ also represents the class $x$. 
    \item $H_1(\Gamma_{1,1}^{1/2}[1];\mathbb{Z})=\mathbb{Z}\{z\}$, where $z$ is represented by $a$. 
    \item Under the inclusion $\Gamma_{1,1}^{1/2}[0] \subset \Gamma_{1,1}$ we have $x \mapsto -2 \tau$ and $y \mapsto \tau$.
    \item Under the inclusion $\Gamma_{1,1}^{1/2}[1] \subset \Gamma_{1,1}$ we have $z \mapsto \tau$.
\end{enumerate}
\end{theorem}

\begin{proof}
Parts (i), (ii) and (iii) immediately imply parts (iv) and (v). 
Moreover, parts (i) and (iii) are equivalent since there is a unique quadratic refinement of Arf invariant $1$ so $\Gamma_{1,1}^{1/2}[1]=\Gamma_{1,1}$.

By \cite[Page 8]{korkmaz} we have the presentation 
$\Gamma_{1,1}=\langle a, b | \; a b a = b a b \rangle$.

Abelianizing this presentation gives part (i). 
We will prove (ii) by finding a presentation for $\Gamma_{1,1}^{1/2}[0]$ and then abelianizing it using GAP. 

The right action of $a, b$ on the set of quadratic refinements of Arf invariant $0$ is given by: 
$a^*(q_{0,0})= q_{0,1}$, $a^*(q_{0,1})=q_{0,0}$, $a^*(q_{1,0})=q_{1,0}$, $b^*(q_{0,0})= q_{1,0}$, $b^*(q_{0,1})=q_{0,1}$, $b^*(q_{1,0})=q_{0,0}$. 
(This is shown by analysing the effect on homology of the corresponding right-handed Dehn twists.)

We will denote $q_{0,0}:=1$, $q_{0,1}:=2$ and $q_{1,0}:=3$, so that $a$ acts on $\{1,2,3\}$ by the permutation $(1 2)$ and $b$ acts on $\{1,2,3\}$ by the permutation $(1 3)$. 

\begin{verbatim}
gap> F:=FreeGroup("a","b");
gap> AssignGeneratorVariables(F);
gap> rel:=[a*b*a*b^-1*a^-1*b^-1];
gap> G:=F/rel;
% This defines group G=\Gamma_{1,1}
gap>  Q:=Group((1,2),(1,3));
gap> hom:=GroupHomomorphismByImages
(G,Q,GeneratorsOfGroup(G),GeneratorsOfGroup(Q));
[ a, b ] -> [ (1,2), (1,3) ]
% permutation representation of G on the elements of Q(1)
of Arf invariant 0. 
gap> S:=PreImage(hom,Stabilizer(Q,1));
% S is the spin mapping class group \Gamma_{1,1}{1/2}[0]
gap> genS:=GeneratorsOfGroup(S);
[ a^-2, b^-2, a*b*a^-1 ] % explicit generators for S.
gap> iso:= IsomorphismFpGroupByGenerators(S,genS);
[ a^-2, b^-2, a*b*a^-1 ] -> [ F1, F2, F3 ]
gap> s:=ImagesSource(iso);
% explicit fp group isomorphic to S via iso, 
so that the generators correspond via iso. 
gap> RelatorsOfFpGroup(s);
[ F3*F1*F3^-1*F2^-1, F3*F2*F3^-1*F2*F1^-1*F2^-1 ] 
% explicit relations for the group S. 
Thus we have a presentation of S. 
gap> AbelianInvariants(s);
[ 0, 0 ] 
% Thus, the abelianization of S is isomorphic to Z \oplus Z. 
gap>  q:=MaximalAbelianQuotient(s);
gap> AbS:=ImagesSource(q);
gap> GeneratorsOfGroup(AbS);
[ f1, f2, f3 ]
gap> RelatorsOfFpGroup(AbS);
[ f1^-1*f2^-1*f1*f2, f1^-1*f3^-1*f1*f3, f2^-1*f3^-1*f2*f3, f1 ]
% Description of the abelianization AbS of S as a f.p. group. 
Thus, f2 and f3 are the free generators of AbS. 
gap> Image(q,s.1);
f2 
% Thus, s.1 corresponds to f2 under abelianization. 
% i.e. a^-2 is one generator.
gap> Image(q,s.3);
f3
% Thus, s.3 corresponds to f3 under abelianization.
% i.e. aba^-1 is the other generator. 
gap> Image(q,s.2)=Image(q,s.1);
true 
% Thus, b^-2 and a^-2 agree in the abelianization. 

\end{verbatim}
\end{proof}

\subsubsection{$g=2$}
Consider the simple closed curves $\alpha_1,\beta_1,\alpha_2,\beta_2$, $\epsilon$ as drawn below, and the corresponding right-handed Dehn twists along them, denoted by 
$a_1, a_2, b_1, b_2, e \in \Gamma_{2,1}$ respectively. 

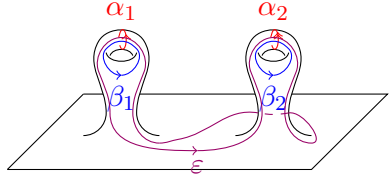
\begin{figure}[H]
    \begin{center}
        \begin{tikzpicture}[scale=1, decoration={
                markings,
                mark=at position 0.5 with {\arrow{<}}}
            ] 
		\draw (0,0) -- (1,1);
		\draw (1,1) -- (5,1) -- (4,0) -- (0,0);
		\churro[x=1, y=0.45]
		\draw[red,dashed] (1.5,1.85) node[above] {$\alpha_1$}
		    to[in=120,out=-120] (1.5,1.6);
		\draw[red,postaction={decorate}] (1.5,1.85) to[in=-30,out=30] (1.5,1.6);
		\draw[blue,->] (1.5,1.25) node[below] {$\beta_1$}
		    to[in=-45,out=0] (1.7,1.6)
		    to[in=0,out=90+45] (1.5,1.7)
		    to[in=45,out=180] (1.3,1.6)
		    to[in=180,out=-90-45]  (1.5,1.25);

		\churro[x=3, y=0.45]
		\draw[red,dashed] (3.5,1.85) node[above] {$\alpha_2$}
		    to[in=120,out=-120] (3.5,1.6);
		    \begin{scope}[red,decoration={
                markings,
                mark=at position 0.4 with {\arrow{<}}}]
		    \draw[postaction={decorate}] (3.5,1.85) to[in=-30,out=30] (3.5,1.6);
		    \end{scope}
		\draw[blue,->] (3.5,1.25) node[below] {$\beta_2$}
		    to[in=-45,out=0] (3.7,1.6)
		    to[in=0,out=90+45] (3.5,1.7)
		    to[in=45,out=180] (3.3,1.6)
		    to[in=180,out=-90-45]  (3.5,1.25);

		\draw[blue!40!red] (2.5,0.25) node[below] {$\varepsilon$}
		    to[out=0,in=-90]  (3.35,0.75)
		    to[out=90,in=-70]  (3.2,1.3)
		    to[out=110,in=180]  (3.5,1.75)
		    to[out=0,in=70]  (3.8,1.3)
		    to[out=-110,in=90]  (3.65,0.75)
		    to[out=-90,in=180]  (4,0.35)
		    to[out=0,in=10]  (3.7,0.75)
		    ;
		\draw[blue!40!red,dashed] (3.7,0.75)
		    to[out=180+10,in=-10]  (3.3,0.75)
		    ;
		\draw[blue!40!red] (3.3,0.75)
		    to[out=180-10,in=10]  (2.25,0.4)
		    to[out=180+10,in=0]  (2,0.35)
		    ;
		    
		\draw[blue!40!red] (1.35,0.75)
		    to[out=90,in=-70]  (1.2,1.3)
		    to[out=110,in=180]  (1.5,1.75)
		    to[out=0,in=70]  (1.8,1.3)
		    to[out=-110,in=90]  (1.65,0.75)
		    to[out=-90,in=180]  (2,0.35);
		    
		\draw[blue!40!red,->] (1.35,0.75)
		    to[out=-90,in=180]  (2.5,0.25)
		    ;		    		    
    \end{tikzpicture}
    \end{center}
\end{figure}

The set of quadratic refinements $Q(2)$ is $\{q_{i_1,j_1,i_2,j_2}: \; i_1,j_1,i_2,j_2 \in \{0,1\}\}$, where $q=q_{i_1,j_1,i_2,j_2}$ satisfies $q(\alpha_1)=i_1$,
 $q(\alpha_2)=i_2$, $q(\beta_1)=j_1$ and $q(\beta_2)=j_2$. 

Now we fix models of $\Gamma_{2,1}^{1/2}[\epsilon]$ via $\Gamma_{2,1}^{1/2}[0]:=\Stab_{\Gamma_{2,1}}(q_{0,0,0,0})$ and $\Gamma_{2,1}^{1/2}[1]:=\Stab_{\Gamma_{2,1}}(q_{1,0,1,1})$.

\begin{theorem}\label{thm: 6.2}
\begin{enumerate}[(i)]

    \item $H_1(\Gamma_{2,1};\mathbb{Z})=\mathbb{Z}/10\{\sigma \cdot\tau\}$, and $\sigma \cdot \tau$ is represented by $a_1$. 
    \item $H_1(\Gamma_{2,1}^{1/2}[0];\mathbb{Z})=\mathbb{Z}\{A\} \oplus \mathbb{Z}/2\{B\}$, where $A$ is represented by $a_1 b_1 a_1^{-1} b_1 b_2 e^{-1}$ and $B$ is represented by $(a_1 b_1 a_1)^2 e b_2^{-1} b_1^{-1}$. 
    \item $H_1(\Gamma_{2,1}^{1/2}[1];\mathbb{Z})=\mathbb{Z}/80\{C\}$, where $C$ is represented by $a_1$. 
    \item Under the inclusion $\Gamma_{2,1}^{1/2}[0] \subset \Gamma_{2,1}$ we have $A \mapsto 2 \sigma \cdot\tau$ and $B \mapsto 5 \sigma \cdot\tau$.
    \item Under the inclusion $\Gamma_{2,1}^{1/2}[1] \subset \Gamma_{2,1}$ we have $C \mapsto \sigma \cdot\tau$.
\end{enumerate}
\end{theorem}

\begin{proof}

We say that the pair $(u,v)$ satisfies the \textit{braid relation} if $u v u= v u v$. 
By \cite[Theorem 2]{wajnryb} there is a presentation 
$$\Gamma_{2,1}=\langle a_1,a_2,b_1,b_2,e | R_1 \sqcup R_2 \sqcup \{(b_1 a_1 e a_2)^5=b_2 a_2 e a_1 b_1^2 a_1 e a_2 b_2\} \rangle,$$ 
where $R_1$ says that $(a_1,b_1)$,$(a_2,b_2)$,$(a_1,e)$,$(a_2,e)$ satisfy the braid relation, and $R_2$ says that each of $(a_1,a_2)$,$(b_1, b_2)$,$(a_1,b_2)$,$(a_2,b_1)$,$(b_1,e)$,$(b_2,e)$ commutes.

Part (i) follows from abelianizing the above presentation of $\Gamma_{2,1}$, and it is compatible with \cite[Lemma 3.6]{E2}.

For part (ii) we will describe the (right) action of $a_1,b_1,a_2,b_2,e$ on the set of $10$ quadratic refinements of Arf invariant $0$, which we will label as $q_{0,0,0,0}:=1$, $q_{0,0,0,1}:=2$, $q_{1,0,0,0}:=3$, $q_{0,0,1,0}:=4$, $q_{1,0,0,1}:=5$, $q_{1,0,1,0}:=6$, $q_{0,1,0,0}:=7$, $q_{0,1,1,0}:=8$, $q_{0,1,0,1}:=9$, $q_{1,1,1,1}:=10$.
The explicit action of each generator as a permutation in $S_{10}$ can be found in the GAP computation below. 
We use GAP to find presentation of $\Gamma_{2,1}^{1/2}[0]$ and its first homology group as follows.
\begin{verbatim}
gap>   F:=FreeGroup("a","x","b","y","c");
% a means a1, x means a2, b means b1, y means a2, c means e. 
gap> AssignGeneratorVariables(F);
gap> rel:= [ a*b*a*b^-1*a^-1*b^-1,  x*y*x*y^-1*x^-1*y^-1, 
a*c*a*c^-1*a^-1*c^-1,  x*c*x*c^-1*x^-1*c^-1,  a*x*a^-1*x^-1, 
a*y*a^-1*y^-1, b*x*b^-1*x^-1, b*y*b^-1*y^-1, b*c*b^-1*c^-1,
y*c*y^-1*c^-1, (y*x*c*a*b^2*a*c*x*y)*(b*a*c*x)^-5];
gap> G:=F/rel;
% This defines group G=\Gamma_{2,1}
gap> Q:=Group((1,7)(2,9)(4,8),(1,2)(3,5)(7,9),(1,3)(2,5)(4,6),
(1,4)(3,6)(7,8),(1,6)(3,4)(9,10));
gap> hom:=GroupHomomorphismByImages
(G,Q,GeneratorsOfGroup(G),GeneratorsOfGroup(Q));
[ a, x, b, y, c ] -> [ (1,7)(2,9)(4,8), (1,2)(3,5)(7,9),
(1,3)(2,5)(4,6), (1,4)(3,6)(7,8), (1,6)(3,4)(9,10) ]
% permutation representation of G on the elements of Q(2)
of Arf invariant 0. 
gap> S:=PreImage(hom,Stabilizer(Q,1));
% This is the spin mapping class group \Gamma_{2,1}{1/2}[0]
gap> genS:=GeneratorsOfGroup(S);
[ a^-2, x^-2, b^-2, y^-2, c^-2, a*b*a^-1, a*c*a^-1, x*y*x^-1,
x*c*x^-1, b*y*c^-1 ] 
% Generators for S: there are 10 of them called s.1,...,s.10.
gap> iso:= IsomorphismFpGroupByGenerators(S,genS);
gap> s:=ImagesSource(iso);
<fp group on the generators [ F1, F2, F3, F4, F5, F6, F7, F8,
F9, F10 ]>
gap> AbelianInvariants(s);
[ 0, 2 ] 
% Thus, the abelianization of S is isomorphic to Z \oplus Z/2. 
gap> q:=MaximalAbelianQuotient(s);
gap> AbS:=ImagesSource(q);
gap> GeneratorsOfGroup(AbS);
[ f1, f2, f3, f4, f5, f6, f7, f8, f9, f10 ]
% this gives a description of the abelianization AbS of S. 
gap> RelatorsOfFpGroup(AbS);
[ f1^-1*f2^-1*f1*f2, f1^-1*f3^-1*f1*f3, f1^-1*f4^-1*f1*f4,
f1^-1*f5^-1*f1*f5, f1^-1*f6^-1*f1*f6, f1^-1*f7^-1*f1*f7,
f1^-1*f8^-1*f1*f8, f1^-1*f9^-1*f1*f9, f1^-1*f10^-1*f1*f10,
  f2^-1*f3^-1*f2*f3, f2^-1*f4^-1*f2*f4, f2^-1*f5^-1*f2*f5,
  f2^-1*f6^-1*f2*f6, f2^-1*f7^-1*f2*f7, f2^-1*f8^-1*f2*f8,
  f2^-1*f9^-1*f2*f9, f2^-1*f10^-1*f2*f10, f3^-1*f4^-1*f3*f4,
  f3^-1*f5^-1*f3*f5, f3^-1*f6^-1*f3*f6, f3^-1*f7^-1*f3*f7,
  f3^-1*f8^-1*f3*f8, f3^-1*f9^-1*f3*f9, f3^-1*f10^-1*f3*f10,
  f4^-1*f5^-1*f4*f5, f4^-1*f6^-1*f4*f6, f4^-1*f7^-1*f4*f7,
  f4^-1*f8^-1*f4*f8, f4^-1*f9^-1*f4*f9, f4^-1*f10^-1*f4*f10,
  f5^-1*f6^-1*f5*f6, f5^-1*f7^-1*f5*f7, f5^-1*f8^-1*f5*f8,
  f5^-1*f9^-1*f5*f9, f5^-1*f10^-1*f5*f10, f6^-1*f7^-1*f6*f7,
  f6^-1*f8^-1*f6*f8, f6^-1*f9^-1*f6*f9, f6^-1*f10^-1*f6*f10,
  f7^-1*f8^-1*f7*f8, f7^-1*f9^-1*f7*f9, f7^-1*f10^-1*f7*f10,
  f8^-1*f9^-1*f8*f9, f8^-1*f10^-1*f8*f10, f9^-1*f10^-1*f9*f10,
  f1, f2, f3, f4, f5, f6, f7, f8, f9^2 ]
% Thus, f9 generates the Z/2 and f10 generates the Z. 
gap> PreImagesRepresentative(q,AbS.9);
(F9*F5^-1)^2*F10^-1
gap> Image(q,(s.9*s.5^-1)^2*s.10^-1)=AbS.9;
true 
% Element of S generating the Z/2 summand, where the Fi index 
the ten generators of S in order. 
F9 and F6 agree on abelianization, and so do F5 and F1. 
Thus, we can replace this generator by (s.6*s.1^-1)^2*s.10^-1, 
which gives B by substituting what s.1,s.6 and s.10 are.
gap> Image(q,s.5^-1*s.10^-1*s.9)=AbS.10;
true
% Generator of the Z summand. This gives A.
\end{verbatim}

Part (iii) is done similarly to part (ii): 
there are 6 quadratic refinements of Arf invariant 1, which we index as: 
$q_{1,0,1,1}:=1$, $q_{0,0,1,1}:=2$, $q_{0,1,1,1}:=3$, $q_{1,1,0,1}:=4$, $q_{1,1,0,0}:=5$, $q_{1,1,1,0}:=6$.
The explicit action of each generator as a permutation on $S_6$ can be found in the GAP computation below. 
The group $G$ in the computation represents $\Gamma_{2,1}$ and it is input in the same way as above. 
\begin{verbatim}
gap> Q:=Group((2,3),(4,5),(1,2),(5,6),(3,4)); 
gap> hom:=GroupHomomorphismByImages
(G,Q,GeneratorsOfGroup(G),GeneratorsOfGroup(Q));
[ a, x, b, y, c ] -> [ (2,3), (4,5), (1,2), (5,6), (3,4) ]
% permutation representation of G on the elements
of Q(2) of Arf invariant 1
gap> S:=PreImage(hom,Stabilizer(Q,1));
% This is the spin mapping class group \Gamma_{2,1}{1/2}[1]
gap> genS:=GeneratorsOfGroup(S);
[ a, x, b^-2, y, c ]
gap> iso:= IsomorphismFpGroupByGenerators(S,genS);
gap> s:=ImagesSource(iso);
gap> AbelianInvariants(s);
[ 5, 16 ] % this shows that the abelianization of s is
isomorphic to Z/5 \oplus Z/16 \cong Z/80. 
gap> q:=MaximalAbelianQuotient(s);
gap> AbS:=ImagesSource(q);
gap> Image(q,s.1);
f1
gap> Image(q,s.1)=AbS.5;
true 
gap> Order(AbS.5);
80 
% This gives the required class $C$. 

\end{verbatim}

Finally, parts (iv) and (v) follow from the explicit description of $A,B,C$ plus using the relations in the abelianization of $\Gamma_{2,1}$. 
\end{proof}

\subsubsection{Stabilizations, Browder brackets and the $Q_{\mathbb{Z}}^1(-)$-operation}

\begin{theorem}\label{thm: 6.3}
\begin{enumerate}[(i)]
    \item $[\sigma,\sigma]=4 \sigma \cdot \tau$. 
    \item $Q_{\mathbb{Z}}^1(\sigma)=3 \sigma \cdot \tau$. 
    \item $x \cdot \sigma_0 = 4 A$ and $y \cdot \sigma_0= 3 A + B$. 
    \item $x \cdot \sigma_1 = 28 C$ and $y \cdot \sigma_1 = C$. 
    \item $z \cdot \sigma_0= C$.
    \item $z \cdot \sigma_1= 3 A+ B$. 
    \item $[\sigma_0,\sigma_0]=-8 A$, $[\sigma_1,\sigma_1]=72 A$ and $[\sigma_0,\sigma_1]= 24 C$. 
    \item $Q_{\mathbb{Z}}^1(\sigma_0)=4A+B$ and $Q_{\mathbb{Z}}^1(\sigma_1)=-36A+B$. 
\end{enumerate}
\end{theorem}

\begin{proof}
Parts (i) and (ii) and appear in \cite[Lemma 3.6]{E2}. 

For part (iii) we use the same GAP computation as in Theorem \ref{thm: 6.2}, (ii). 
Right stabilization by $\sigma_0$ sends $q_{0,0} \mapsto q_{0,0,0,0}$ and $a \mapsto a_1$, $b \mapsto b_1$. 
Therefore, $x \cdot \sigma_0$ is represented by $a_1^{-2}$ and $y \cdot \sigma_0$ is represented by $a_1 b_1 a_1^{-1}$

\begin{verbatim}
gap> Image(q,s.1);
f1*f9^-2*f10^4
gap> Image(q,s.6);
f6*f9^-3*f10^3
\end{verbatim}

This says that under abelianization $a_1^{-2}$ (which is $s.1$) is mapped to $f1*f9^{-2}*f10^4$, which is $4A$ by the GAP computation in the proof of Theorem \ref{thm: 6.2}, (ii). 
Also, since $s.6$ means $a_1 b_1 a_1^{-1}$ then we get $y \cdot \sigma_0= 3A+B$. 

Proof of (iv):

Observe that when we stabilize by $- \cdot \sigma_1$ we send $q_{0,0} \mapsto q_{0,0,1,1}$, whereas our choice of quadratic refinement of Arf invariant $1$ is $q_{0,1,1,1}$. 
To fix this we will use conjugation: $b_1 \in \Gamma_{2,1}$ satisfies $b_1^*(q_{1,0,1,1})=q_{0,0,1,1}$ and so 
$$\Stab_{\Gamma_{2,1}}(q_{0,0,1,1}) \xrightarrow{b_1 \cdot - \cdot b_1^{-1}} \Stab_{\Gamma_{2,1}}(q_{1,0,1,1})$$
is an isomorphism. 

This is non-canonical, but its action in group homology is canonical: 
If $u \in \Gamma_{2,1}$ satisfies $u^*(q_{1,0,1,1})=q_{0,0,1,1}$ then the maps $u \cdot - \cdot u^{-1}$ and $b_1 \cdot - \cdot b_1^{-1}$ differ by conjugation by $b_1 u^{-1} \in \Stab_{\Gamma_{2,1}}(q_{1,0,1,1})$, which acts trivially in group homology.
Thus, for homology computations, $x \cdot \sigma_1$ is represented by $b_1 a_1^{-2} b_1^{-1} \in \Gamma_{2,1}^{1/2}[1]$ and $y \cdot \sigma_1$ is represented by $b_1 a_1 b_1 a_1^{-1} b_1^{-1}$. 

Now we use GAP computation in the proof of Theorem \ref{thm: 6.2},(iii) to see where these elements map
\begin{verbatim}
gap> Image(iso,G.3*G.1^-2*G.3^-1);
F1^-1*F3*F1 
% Expression for b1*a1^-1*b1^-1 \in S in terms of its generators. 
gap>  Image(q,Image(iso,G.3*G.1^-2*G.3^-1))=Image(q,s.1)^28;
true 
% This shows that indeed b1*a1^-2*b1^-1 is the element 28
in the abelianization. 
gap> Image(iso,G.3*G.1*G.3*G.1^-1*G.3^-1);
F1
gap> Image(q, Image(iso,G.3*G.1*G.3*G.1^-1*G.3^-1))=Image(q,s.1);
true
% This shows that b_1 a_1 b_1 a_1^{-1} b_1^{-1} is the element 1. 
\end{verbatim}

Part (v) is similar to the previous part:
right stabilization by $\sigma_0$ sends $q_{1,1} \mapsto q_{1,1,0,0}$, and so we need to conjugate by $b_1 a_1 e a_2 \cdot - \cdot (b_1 a_1 e a_2)^{-1}$ to go to $\Gamma_{2,1}^{1/2}[0]$. 
Also, $z$ is represented by $a$, so $z \cdot \sigma_0$ is represented by $b_1 a_1 e a_2 a_1 (b_1 a_1 e a_2)^{-1}$.

Using GAP: 
\begin{verbatim}
gap> Image(iso,(G.3*G.1*G.5*G.2)*G.1*(G.3*G.1*G.5*G.2)^-1);
F3^-1*F5*F3
gap> Image(q, Image(iso,(G.3*G.1*G.5*G.2)*G.1*(G.3*G.1*G.5*G.2)^-1))
=Image(q,s.1);
true
\end{verbatim}

Part (vi) follows from the following GAP computation: 
\begin{verbatim}
gap> Image(iso,(G.1*G.2*G.5)*G.1*(G.1*G.2*G.5)^-1);
F9
gap> Image(q,Image(iso,(G.1*G.2*G.5)*G.1*(G.1*G.2*G.5)^-1))
=Image(q,s.9);
true
\end{verbatim}

To prove part (vii) we will need the following claim
\begin{claim}
The element $-[\sigma,\sigma] \in \Gamma_{2,1}$ is represented by $(b_2 a_2 e a_1 b_1)^6 (a_1 b_1)^6 (a_2 b_2)^{-6}$
\end{claim}

\begin{proof}
By \cite[Lemma 3.6, Figure 4]{E2} we can write $-[\sigma,\sigma]$ as $t_w t_u^{-1} t_v^{-1}$, where $u,v,w$ are the following curves called ``a'', ``b'' and ``c'' respectively in \cite{E2}, 
and $t_w, t_u, t_v$ are the corresponding right-handed Dehn twists along them. 
Now we use \cite[Lemma 21 (iii)]{wajnryb} to write each of $t_u,t_v,t_w$ in terms of the generators, yielding the following:
$t_u=(a_1 b_1)^6$, $t_v=(a_2 b_2)^6$ and $t_w=(b_2 a_2 e a_1 b_1)^6$. 
\end{proof}

The above element lies in $\Stab_{\Gamma_{2,1}}(q)$ for any quadratic refinement $q$ because each of the curves $u,v,w$ is disjoint from the curves $\alpha_1,\beta_1,\alpha_2,\beta_2$, and hence $t_u,t_v,t_w$ do not affect the value of $q$ along the standard generators of $H_1(\Sigma_{2,1};\mathbb{Z})$.
Thus, $[\sigma_i,\sigma_j] \in H_1(B\Gamma_{2,1}^{1/2}[i+j (\mod 2)];\mathbb{Z})$ is represented by 
$$(b_2 a_2 e a_1 b_1)^6 (a_1 b_1)^6 (a_2 b_2)^{-6} \in \Stab_{\Gamma_{2,1}}(q_{i,i,j,j}),$$
and then we conjugate this element so that it lies in our fixed choices of stabilizers. 

For $-[\sigma_0,\sigma_0]$ we don't need to conjugate, so we find 
\begin{verbatim}
gap> Image(q, Image(iso,(G.4*G.2*G.5*G.1*G.3)^6*(G.1*G.3)^-6
*(G.2*G.4)^-6))=AbS.10^8;
true
% this means that abelianization is (8,0) \in Z \oplus Z/2.
\end{verbatim}

For $-[\sigma_1,\sigma_1]$ need to conjugate by $a_1 a_2 e$, and we find
\begin{verbatim}
gap>  Image(q,Image(iso,G.1*G.2*G.5*(G.4*G.2*G.5*G.1*G.3)^6*
(G.1*G.3)^-6*(G.2*G.4)^-6*(G.1*G.2*G.5)^-1))=AbS.10^-72;
true
\end{verbatim}

For $-[\sigma_0,\sigma_1]$ we need to conjugate by $b_1$, and we get
\begin{verbatim}
gap>  Image(q,Image(iso,G.3*(G.4*G.2*G.5*G.1*G.3)^6*
(G.1*G.3)^-6*(G.2*G.4)^-6*G.3^-1))=AbS.5^56;
true 
\end{verbatim}

Finally, to prove (viii) we use that $2 Q_{\mathbb{Z}}^1(\sigma_{\epsilon}) = -[\sigma_{\epsilon},\sigma_{\epsilon}]$, which follows from the discussion in \cite[Page 9]{E2}. 

For $\epsilon=0$, Theorem \ref{thm: 6.2} together with part (vii) of this theorem say that $Q_{\mathbb{Z}}^1(\sigma_0)$ is either $4A$ or $4A+B$. 
But, by part (ii) of this theorem we know that it must map to $3 \sigma \cdot \tau \in H_1(\Gamma_{2,1};\mathbb{Z})$. 
Using Theorem \ref{thm: 6.2},(iv) we get that the answer must be $4A+B$. 
The computation of the case $\epsilon=1$ is similar. 
\end{proof}

\subsubsection{$g=3$}

\begin{theorem}\label{thm: 6.4}
\begin{enumerate}[(i)]
    \item $H_1(\Gamma_{3,1}^{1/2}[0];\mathbb{Z}) \cong \mathbb{Z}/4$, where $A \cdot \sigma_0 = y \cdot \sigma_0^2$ is a generator and $B \cdot \sigma_0= 2 A \cdot \sigma_0$.  
    \item $Q_{\mathbb{Z}}^1(\sigma_0) \cdot \sigma_0= 2 y \cdot \sigma_0^2$. 
    \item $H_1(\Gamma_{3,1}^{1/2}[1];\mathbb{Z}) \cong \mathbb{Z}/4$, where $ y \cdot \sigma_0 \cdot \sigma_1= z \cdot \sigma_0^2$ is a generator.  
    \item $Q_{\mathbb{Z}}^1(\sigma_0) \cdot \sigma_1= 2 y \cdot \sigma_0 \cdot \sigma_1=Q_{\mathbb{Z}}^1(\sigma_1) \cdot \sigma_0$ and $Q_{\mathbb{Z}}^1(\sigma_0) \cdot \sigma_0=Q_{\mathbb{Z}}^1(\sigma_1) \cdot \sigma_1$.
\end{enumerate}
\end{theorem}

\begin{proof}
By \cite[Theorem 1]{wajnryb} we get a presentation of $\Gamma_{3,1}$ with generators
$a_1,a_2,a_3,b_1,b_2,b_3,e_1,e_2$, where the $a_i,b_i$ are defined as in the cases $g=1,2$, $e_1$ is what was called $e$ in the $g=2$ case using the first two handles, and $e_2$ is defined analogously, but using the second and third handles instead.

To prove (i) and (ii) we fix our quadratic refinement of Arf invariant $0$ to be the one evaluating to $0$ on all the $\alpha_i$'s and $\beta_i$'s. 
In this case the strategy to find a presentation for $\Gamma_{3,1}^{1/2}[0]$ is different: instead of computing the action on quadratic refinements we will write down elements of $\Gamma_{3,1}^{1/2}[0]$ (inspired by expressions from previous computations) and check that the subgroup they generate has index $36$ inside $\Gamma_{3,1}$, and hence that it must agree with $\Gamma_{3,1}^{1/2}[0]$. 

\begin{verbatim}
gap> F:=FreeGroup("a1","a2","a3","b1","b2","e1","e2");
gap> AssignGeneratorVariables(F);
gap> rel:=[a1*b1*a1*b1^-1*a1^-1*b1^-1, a2*b2*a2*b2^-1*a2^-1*b2^-1,
a1*e1*a1*e1^-1*a1^-1*e1^-1, a2*e1*a2*e1^-1*a2^-1*e1^-1,
a2*e2*a2*e2^-1*a2^-1*e2^-1, a3*e2*a3*e2^-1*a3^-1*e2^-1,
a1*a2*a1^-1*a2^-1, a1*a3*a1^-1*a3^-1, a3*a2*a3^-1*a2^-1,
b1*b2*b1^-1*b2^-1, a1*b2*a1^-1*b2^-1, a2*b1*a2^-1*b1^-1,
a3*b2*a3^-1*b2^-1, a3*b1*a3^-1*b1^-1, b1*e1*b1^-1*e1^-1,
b2*e1*b2^-1*e1^-1, b1*e2*b1^-1*e2^-1, b2*e2*b2^-1*e2^-1,
a1*e2*a1^-1*e2^-1, a3*e1*a3^-1*e1^-1,  e1*e2*e1^-1*e2^-1,
(b1*a1*e1*a2)^-5*b2*a2*e1*a1*b1^2*a1*e1*a2*b2, ((b2*a2*e1*b1^-1)*
(e2*a2*a3*e2)* (a2*e1*a1*b1)^-1*b2*(a2*e1*a1*b1) *(e2*a2*a3*e2)^-1*
(b1*e1^-1*a2^-1*b2^-1)*a1*a2*a3)^-1*(a2*e1*a1*b1)^-1*b2*
(a2*e1*a1*b1)* (e2*a2*a3*e2)* (a2*e1*a1*b1)^-1*b2*(a2*e1*a1*b1)
*(e2*a2*a3*e2)^-1*(e1*a1*a2*e1)*(e2*a2*a3*e2)*(a2*e1*a1*b1)^-1*
b2*(a2*e1*a1*b1) *(e2*a2*a3*e2)^-1*(e1*a1*a2*e1)^-1];
% this encodes the presentation. 
gap> G:=F/rel;
% this is \Gamma_{3,1}
gap> H:=Subgroup(G,[G.1^-2,G.2^-2,G.3^-2,G.4^-2,G.5^-2,G.6^-2,
G.7^-2,G.4*G.5*G.6^-1,G.5*G.2*G.5^-1,G.4*G.1*G.4^-1,G.2*G.7*G.2^-1,
G.7*G.3*G.7^-1]);
% H lies inside Stab_{\Gamma_{3,1}}(q_{000000}) as each
generator of H fixes q_{000000}.
gap> Index(G,H);
36 
% Thus, H=\Gamma_{3,1}^{1/2}[0] as there are 36 elements 
in Q(3) of Arf invariant 0.
gap> AbelianInvariants(H);
[ 4 ]
% This gives H_1(\Gamma_{3,1}^{1/2}[0]) \cong Z/4. 
gap> genH:=GeneratorsOfGroup(H);
gap> iso:=IsomorphismFpGroupByGenerators(H,genH);
gap> S:=ImagesSource(iso); 
% f.p. group isomorphic to H. 
gap> q:=MaximalAbelianQuotient(S);
gap> AbS:=ImagesSource(q);
gap> Order(Image(q,Image(iso,G.1*G.4*G.1^-1)));
4 % Thus, y*\sigma_0^2 is a generator of H_1(\Gamma_{3,1}{1/2}[0])
gap> Order( Image(q,Image(iso,(G.1*G.4*G.1)^2*G.6*G.5^-1*G.4^-1)));
2 
% Thus B* \sigma_0 =2 mod 4 by definition of B.
\end{verbatim}

Now, to finish we need to check two things:
The first one is that $A \cdot \sigma_0= y \cdot \sigma_0^2$ is a generator: 
By Theorem \ref{thm: 6.3}(ii) we have $y \cdot \sigma_0= 3A+B$ so using the above GAP computations we find $A \cdot \sigma_0= y \cdot \sigma_0^2$. 

By Theorem \ref{thm: 6.3} (viii) we have
$Q_{\mathbb{Z}}^1(\sigma_0) \cdot \sigma_0= 4 A \cdot \sigma_0+ B \cdot \sigma_0= B \cdot \sigma_0$, as required. 

To prove parts (iii) and (iv) we fix our quadratic refinement of Arf invariant 1 to be the one with value 1 in all the $a_i$ and $b_i$.
Now we use GAP (F,G are as before, so we will not copy that part again) and a similar idea as above to get the result. 
\begin{verbatim}
gap> H:=Subgroup(G,[G.1,G.2,G.3,G.4,G.5,G.6^-2,G.7^-2,G.6*G.4*G.6^-1,
G.6*G.5*G.6^-1,G.7*G.5*G.7^-1,(G.6*G.2*G.1)*G.3*(G.6*G.2*G.1)^-1,
(G.6*G.2*G.1)*G.7*(G.6*G.2*G.1)^-1, (G.6*G.2*G.1)*(G.6*G.5*G.4)*
(G.6*G.2*G.1)^-1,(G.7*G.3*G.2)*G.1*( G.7*G.3*G.2)^-1,
( G.7*G.3*G.2)*G.6*( G.7*G.3*G.2)^-1]);
% H lies inside Stab_{\Gamma_{3,1}}(q_{111111}) as each
generator of H fixes q_{111111}.
gap> Index(G,H);
28 %Thus, H=\Gamma_{3,1}{1/2}[1] as there are 28 elements in 
Q(3) of Arf invariant 1. 
gap> AbelianInvariants(H);
[ 4 ]
% this gives H_1(\Gamma_{3,1}{1/2}[1])=Z/4. 
\end{verbatim}

This shows that $H_1(\Gamma_{3,1}^{1/2}[1];\mathbb{Z}) \cong \mathbb{Z}/4$. 

By Theorem \hyperref[theorem A]{A} (i), the map $\sigma_{\epsilon} \cdot - : H_1(\Gamma_{g-1,1}^{1/2}[\delta-\epsilon];\mathbb{Z}) \rightarrow H_1(\Gamma_{g,1}^{1/2}[\delta];\mathbb{Z})$ is surjective for $g \geq 4$ and any $\epsilon, \delta$. 
(The proof of Theorem \hyperref[theorem A]{A} uses the results of the Appendix, but part (i) is shown independently of these first homology computations.)

Moreover, the stable values $H_1(\Gamma_{\infty,1}^{1/2}[\delta];\mathbb{Z})$ are both isomorphic to $\mathbb{Z}/4$ by \cite[Theorem 1.4]{Picardspin} plus \cite[Theorem 2.14]{rspin}.
Thus the groups $H_1(\Gamma_{g,1}^{1/2}[\delta];\mathbb{Z})$ are stable for any $g \geq 3$ and any $\delta \in \{0,1\}$. 

Since $\sigma_0^2=\sigma_1^2$ then $(Q_{\mathbb{Z}}^1(\sigma_0) \cdot \sigma_1) \cdot \sigma_1 = Q_{\mathbb{Z}}^1(\sigma_0) \cdot \sigma_0^2= 2 y \cdot \sigma_0^3 = (2 y \sigma_0 \cdot \sigma_1) \cdot \sigma_1$, and by the above  stability result $Q_{\mathbb{Z}}^1(\sigma_0) \cdot \sigma_1= 2 y \cdot \sigma_0 \cdot \sigma_1$. 

Also, $ y \cdot \sigma_0^3 = (y  \cdot \sigma_0 \cdot \sigma_1) \cdot \sigma_1$ is a generator of $H_1(\Gamma_{4,1}^{1/2}[0];\mathbb{Z})$ by the stability  plus part (i) of this theorem. 
Thus, $y \cdot \sigma_0 \cdot \sigma_1$ is a generator of $H_1(\Gamma_{3,1}^{1/2}[1];\mathbb{Z})$ by applying stability.
Using that $z \cdot \sigma_0= y \cdot \sigma_1$ (by Theorem \ref{thm: 6.3}) we find $z \cdot \sigma_0^2= y \cdot \sigma_1 \cdot \sigma_0$. 

Finally, by Theorem \ref{thm: 6.3}, $Q_{\mathbb{Z}}^1(\sigma_1)-Q_{\mathbb{Z}}^1(\sigma_0)=-40 A$, so any stabilization of this vanishes because it lives in a $4$-torsion group. 
\end{proof}

\subsection{Quadratic symplectic groups over $\mathbb{Z}$} \label{appendix symplectic}

The proofs of this section will be very similar to the ones of Section \ref{appendix mcg}, but using the explicit presentations of $Sp_{2g}(\mathbb{Z})$ given in \cite{presentationsymplectic}. 
The computation in Theorems \ref{thm: 6.6}, \ref{thm: 6.7} and \ref{thm: 6.8} about the first homology of the quadratic symplectic groups of Arf invariant $1$ is used in \cite[Section 4.1]{framings}.

\begin{rem}
In \cite{presentationsymplectic} they write matrices using a different basis.   
We will change the matrices given \cite{presentationsymplectic} to our choice of basis of Section \ref{section symplectic} without further notice in all the following computations. 

\end{rem}

\subsubsection{$g=1$} \label{11.2.1}

\begin{theorem}\label{thm: 6.6}
\begin{enumerate}[(i)]
    \item $H_1(Sp_2(\mathbb{Z});\mathbb{Z})=\mathbb{Z}/12\{t\}$, where $t$ is represented by $\begin{psmallmatrix} 1 & 1 \\ 0 & 1 \end{psmallmatrix} \in Sp_2(\mathbb{Z})$. 
    \item $H_1(Sp_2^0(\mathbb{Z});\mathbb{Z})=\mathbb{Z}\{\mu\} \oplus \mathbb{Z}/4\{\lambda\}$, where $\mu$ is represented by $\begin{psmallmatrix} 1 & 2 \\ 0 & 1 \end{psmallmatrix} \in Sp_2^0(\mathbb{Z})$ and $\lambda$ is represented by $\begin{psmallmatrix} 0 & -1 \\ 1 & 0 \end{psmallmatrix} \in Sp_2^0(\mathbb{Z})$. 
    \item $H_1(Sp_2^1(\mathbb{Z});\mathbb{Z})=\mathbb{Z}/12\{t'\}$, where $t'$ is represented by $\begin{psmallmatrix} 1 & 1 \\ 0 & 1 \end{psmallmatrix} \in Sp_2^1(\mathbb{Z})$.
\end{enumerate}
\end{theorem}

\begin{proof}
By \cite[Theorem 1]{presentationsymplectic} we have 
$$Sp_2(\mathbb{Z})= \langle L,N | (LN)^2=N^3, N^6=1 \rangle$$
where
$L= \begin{psmallmatrix} 1 & 1 \\ 0 & 1 \end{psmallmatrix}$
and
$N= \begin{psmallmatrix} 0 & 1 \\ -1 & 1 \end{psmallmatrix}$. 

We will use the same notation as in Section \ref{appendix mcg} for the quadratic refinements, where now $\alpha,\beta$ are be the standard hyperbolic basis of $(\mathbb{Z}^2,\Omega_1)$. 
We let
$Sp_2^0(\mathbb{Z}):= \Stab_{Sp_2(\mathbb{Z})}(q_{0,0})$ and $Sp_2^1(\mathbb{Z}):= \Stab_{Sp_2(\mathbb{Z})}(q_{1,1})$.
We then compute the action of $L, N$ on the set of quadratic refinements of each invariant (see the GAP formulae below). 

Since there is a unique quadratic refinement of Arf invariant 1 then $Sp_2^1(\mathbb{Z})=Sp_2(\mathbb{Z})$ so parts (i) and (iii) are equivalent.  
Thus, it suffices to show parts (i) and (ii). 
To prove (i) we abelianize the presentation of $Sp_2(\mathbb{Z})$ to get $\mathbb{Z}/12\{L\}$. 
To prove (ii) we use GAP
\begin{verbatim}
gap> F:=FreeGroup("L","N");
gap> AssignGeneratorVariables(F);
gap> rel:=[(L*N)^2*N^-3, N^6];
gap> G:=F/rel;
gap> Q:=Group((1,2),(1,2,3));
gap> hom:=GroupHomomorphismByImages
(G,Q,GeneratorsOfGroup(G),GeneratorsOfGroup(Q));
[ L, N ] -> [ (1,2), (1,2,3) ]
gap> S:=PreImage(hom,Stabilizer(Q,1));
% This is Sp_2^0(Z)
gap> AbelianInvariants(S);
[ 0, 4 ]
gap> genS:=GeneratorsOfGroup(S);
[ L^-2, N*L^-1 ]
gap> iso:=IsomorphismFpGroupByGenerators(S,genS);
gap> s:=ImagesSource(iso);
gap> q:=MaximalAbelianQuotient(s);
[ F1, F2 ]  -> [ f2, f_1^-1*f2 ]
gap> AbS:=ImagesSource(q);
gap> GeneratorsOfGroup(AbS);
[ f1, f2 ]
gap> RelatorsOfFpGroup(AbS);
[ f1^-1*f2^-1*f1*f2, f1^4 ]
\end{verbatim}

From these we get that $L^2$ is a generator of the $\mathbb{Z}$ summand. 
Moreover, $N L^{-1} L^2= N L$ maps to a generator of the $\mathbb{Z}/4$ summand, and this matrix is precisely the conjugation by $\Omega_1$ of our choice of matrix for $\lambda$. 
\end{proof}

\subsubsection{$g=2$}

\begin{theorem}\label{thm: 6.7}
\begin{enumerate}[(i)]
    \item $H_1(Sp_4(\mathbb{Z});\mathbb{Z})=\mathbb{Z}/2\{t \cdot \sigma\}$, where $\sigma, t$ are as in Theorem \ref{thm: 6.6}.
    \item $H_1(Sp_4^0(\mathbb{Z});\mathbb{Z})=\mathbb{Z}/2\{Q_{\mathbb{Z}}^1(\sigma_0)\} \oplus \mathbb{Z}/4\{\lambda \cdot \sigma_0\}$, and $Q_{\mathbb{Z}}^1(\sigma_0)$ is represented by $\begin{psmallmatrix} 0 & 0 & 1 & 0 \\ 0 & 0 & 0 & 1 \\ 1 & 0 & 0 & 0 \\ 0 & 1 & 0 & 0 \end{psmallmatrix} \in Sp_4^0(\mathbb{Z})$. 
    
    Moreover, $\mu \cdot \sigma_0=0$. 
    
    \item $H_1(Sp_4^1(\mathbb{Z});\mathbb{Z})=\mathbb{Z}/4\{t' \cdot \sigma_0\}$, where $t'$ is as in Theorem \ref{thm: 6.6}.
    
    \item $t' \cdot \sigma_1= \lambda \cdot \sigma_0$, $\mu \cdot \sigma_1= 0$, $\lambda \cdot \sigma_1= t' \cdot \sigma_0$, $Q_{\mathbb{Z}}^1(\sigma_0)=Q_{\mathbb{Z}}^1(\sigma_1)$ and $[\sigma_0,\sigma_1]=0$.
\end{enumerate}
\end{theorem}

\begin{proof}
By \cite[Theorem 2]{presentationsymplectic} $Sp_4(\mathbb{Z})$ has a presentation with two generators $L, N$ (see the GAP computations below for the relations), where $L$ is given by the stabilization of the matrix called $L$ in Section \ref{11.2.1}. 

To prove (i) we compute
\begin{verbatim}
gap> F:=FreeGroup("L","N");
gap> AssignGeneratorVariables(F);
gap> rel:=[N^6, (L * N)^5, (L *N^-1)^10, (L* N^-1* L * N)^6, 
L *(N^2*L*N^4)* L^-1 * (N^2*L*N^4)^-1, L *(N^3*L*N^3)* L^-1 *
(N^3*L*N^3)^-1, L *(L*N^-1)^5* L^-1 * (L*N^-1)^-5];
gap> G:=F/rel;
% This is Sp_4(Z)
gap> p:=MaximalAbelianQuotient(G);
[ L, N ] -> [ f1, f1 ]
gap> AbG:=ImagesSource(p);
<pc group of size 2 with 2 generators>
% This says H_1(Sp_4(Z)) is a group with 2 elements.
gap> Order(Image(p,G.1));
2
%This gives the required generator: L
\end{verbatim}

To prove part (ii) we add more GAP computations to the above, using a permutation representation of how $L,N$ act on the 10 quadratic refinements of Arf invariant 0 (we use same indexing as in the proof of Theorem \ref{thm: 6.2}, and action is computed similarly). 
\begin{verbatim}
gap> Q:=Group((1,2)(4,6)(5,8),(2,3,4,5,6,7)(8,9,10));
gap> hom:=GroupHomomorphismByImages
(G,Q,GeneratorsOfGroup(G),GeneratorsOfGroup(Q));
% Permutation representation of G on the 10 quadratic refinements
of Arf invariant 0. 
gap> S:=PreImage(hom,Stabilizer(Q,1));
% This is the group Sp_4^0(Z)
gap> genS:=GeneratorsOfGroup(S);
[ L^-2, N, L*N*L*N^-1*L^-1, L*N^-1*L*N*L^-1 ]
gap> iso:= IsomorphismFpGroupByGenerators(S,genS);
gap> s:=ImagesSource(iso);
gap> q:=MaximalAbelianQuotient(s);
gap> AbS:=ImagesSource(q);
gap> AbelianInvariants(S);
[ 2, 4 ]
gap> Order(Image(q,s.1));
1
% By definition of \mu it is represented by L^2, 
so its stabilization vanishes.
gap> Order(Image(q,s.2));
2
gap> Order(Image(q,s.4));
4
gap> Image(q,s.2)=Image(q,s.4)^2;
false
% These last computations say that L*N^-1*L*N*L^-1 generates
the Z/4 summand, and that N generates the Z/2 summand. 
\end{verbatim}

By \cite[Theorem 2]{presentationsymplectic}, he matrix $N$ is given by 
N=$\begin{psmallmatrix} 0 & 1 & -1 & 0 \\ -1 & 0 & 0 & 0 \\ -1 & 0 & 0 & 1 \\ 0 & 0 & -1 & 0 \end{psmallmatrix}$.

Thus,
$N^3= \begin{psmallmatrix} 0 & 0 & 1 & 0 \\ 0 & 0 & 0 & 1 \\ 1 & 0 & 0 & 0 \\ 0 & 1 & 0 & 0 \end{psmallmatrix} \in Sp_4^0(\mathbb{Z})$ represents $Q_{\mathbb{Z}}^1(\sigma_0)$ because it represents $Q_{\mathbb{Z}}^1(\sigma)$ and it stabilizes the quadratic refinement $q_{0,0,0,0}$, so this generates the $\mathbb{Z}/2$ summand. 

Also $L N^{-1} L N L= \begin{psmallmatrix} 0 & 1 & 0 & 0 \\ -1 & 0 & 0 & 0 \\ 0 & 0 & 1 & 0 \\ 0 & 0 & 0 & 1 \end{psmallmatrix}$, which by the last paragraph of the proof of Theorem \ref{thm: 6.6} is the stabilization of the matrix $\lambda$ conjugated by $\Omega_1$. 
Since $\Omega_1 \in Sp_{4}^0(\mathbb{Z})$ then $L N^{-1} L N L$ represents the homology class $\lambda \cdot \sigma_0$. 
By the GAP computations $L N^{-1} L N L= L N^{-1} L N L^{-1} L^2$ is a generator of the $\mathbb{Z}/4$ summand, as required.  

To prove part (iii) we also use the same GAP program but this time we compute the permutation representation on the quadratic refinements of Arf invariant 1. 
We will pick our quadratic refinement of Arf invariant 1 to be $q_{1,1,0,0}$.
\begin{verbatim}
gap> T:=Group((3,4),(1,2,3,4,5,6));
gap> homtwo:=GroupHomomorphismByImages
(G,T,GeneratorsOfGroup(G),GeneratorsOfGroup(T));
% Permutation representation of G on the 6 quadratic refinements
of Arf invariant 1 indexed so that q_{1,1,0,0}=1. 
gap> SS:=PreImage(homtwo,Stabilizer(T,1));
% This is Sp_4^1(Z)
gap> genSS:=GeneratorsOfGroup(SS);
[ L, N*L*N^-1, N^-1*L*N, N^2*L^-2*N^-2, N^3*L^-1*N^-2 ]
gap> isotwo:=IsomorphismFpGroupByGenerators(SS,genSS);
gap> ss:=ImagesSource(isotwo);
gap> qq:=MaximalAbelianQuotient(ss);
gap> AbSS:=ImagesSource(qq);
<pc group of size 4 with 5 generators>
% This says that H_1(Sp_4^1(Z)) is a group of order 4
gap> Order(Image(qq,ss.1));
4 
% This shows the group is cyclic and gives the claimed generator.
\end{verbatim}

To prove (iv) we use the $E_2$-algebra map from the $E_2$-algebra of spin mapping class groups to the one of quadratic symplectic groups, which is induced by the obvious functor $\MCG \rightarrow \Sp$ and the fact that the quadratic refinements functor $Q$ is essentially the same in both cases. 
In more concrete terms, the functor just sends the spin mapping class groups to their actions on first homology, which are quadratic symplectic groups. 

The Dehn twist $a \in \Gamma_{1,1}$ maps to the matrix $\begin{psmallmatrix} 1 & 1 \\ 0 & 1 \end{psmallmatrix} \in Sp_2(\mathbb{Z})$, and the Dehn twist $b \in \Gamma_{1,1}$ maps to $\begin{psmallmatrix} 1 & 0 \\ -1 & 1 \end{psmallmatrix} \in Sp_2(\mathbb{Z})$.
Thus, $a^{-2} \mapsto \begin{psmallmatrix} 1 & -2 \\ 0 & 1 \end{psmallmatrix}= \begin{psmallmatrix} 1 & 2 \\ 0 & 1 \end{psmallmatrix}^{-1}$ and $a b a^{-1} \mapsto \begin{psmallmatrix} 0 & 1 \\ -1 & 2 \end{psmallmatrix}= \begin{psmallmatrix} 0 & 1 \\ -1 & 0 \end{psmallmatrix} \cdot \begin{psmallmatrix} 1 & 2 \\ 0 & 1 \end{psmallmatrix}^{-1}$
By Theorems \ref{thm: 6.1}, \ref{thm: 6.2} and \ref{thm: 6.3} we get $x \mapsto -\mu$, $y \mapsto  \lambda - \mu$ and $z \mapsto t'$.
Also by definition $\sigma_{\epsilon} \mapsto \sigma_{\epsilon}$ for $\epsilon \in \{0,1\}$. 
Thus, by Theorem \ref{thm: 6.3} we get: $x \cdot \sigma_1= 28 z \cdot \sigma_0$ and so $- \mu \cdot \sigma_1= 28 t' \cdot \sigma_0 = 0$. 
Also, $y \cdot \sigma_1= z \cdot \sigma_0$ so $(\lambda-\mu) \cdot \sigma_1= t' \cdot \sigma_0$, giving the result. 
Furthermore, $z \cdot \sigma_1= y \cdot \sigma_0$ so $t' \cdot \sigma_1= (\lambda- \mu) \cdot \sigma_0$, hence giving the result. 
Finally, $Q_{\mathbb{Z}}^1(\sigma_1)=Q_{\mathbb{Z}}^1(\sigma_0)-10 x \cdot \sigma_0$, so $Q_{\mathbb{Z}}^1(\sigma_1)=Q_{\mathbb{Z}}^1(\sigma_0)+ 10 \mu \cdot \sigma_0= Q_{\mathbb{Z}}^1(\sigma_0)$, and $[\sigma_0,\sigma_1]=24 z \cdot \sigma_0 \mapsto 0$. 
\end{proof}

\subsubsection{$g=3$}

\begin{theorem}\label{thm: 6.8}
\begin{enumerate}[(i)]
\item $H_1(Sp_{6}^0(\mathbb{Z};\mathbb{Z}) = \mathbb{Z}/4\{\lambda \cdot \sigma_0^2\}$. 
\item $Q_{\mathbb{Z}}^1(\sigma_0) \cdot \sigma_0 = 2 \lambda \cdot \sigma_0^2$. 
\item $Q_{\mathbb{Z}}^1(\sigma_0) \cdot \sigma_0= Q_{\mathbb{Z}}^1(\sigma_1) \cdot \sigma_1$ and $Q_{\mathbb{Z}}^1(\sigma_0) \cdot \sigma_1= Q_{\mathbb{Z}}^1(\sigma_1) \cdot \sigma_0 = 2 \lambda \cdot \sigma_0 \cdot \sigma_1$. 
\item $H_1(Sp_{6}^1(\mathbb{Z});\mathbb{Z}) = \mathbb{Z}/4\{\lambda \cdot \sigma_0 \cdot \sigma_1\}$ and $\lambda \cdot \sigma_0 \cdot \sigma_1= t' \cdot \sigma_0^2$. 
\end{enumerate}
\end{theorem}

\begin{proof}

By Theorem \ref{thm: 6.4}(i) $H_1(\Gamma_{3,1}^{1/2}[0];\mathbb{Z})=\mathbb{Z}/4\{y \cdot \sigma_0^2\}$. 
The homomorphism $\Gamma_{3,1}^{1/2}[0] \rightarrow Sp_6^0(\mathbb{Z})$ is surjective because $\Gamma_{3,1} \rightarrow Sp_6(\mathbb{Z})$ is, and hence $\mathbb{Z}/4\{y \cdot \sigma_0^2\}$ surjects onto $H_1(Sp_6^0(\mathbb{Z});\mathbb{Z})$. 
Using the $E_2$-algebra map of the previous section $y \cdot \sigma_0^2 \mapsto \lambda \cdot \sigma_0^2$. 
This gives part (ii) by Theorem \ref{thm: 6.4}(ii). 
The rest of part (i) follows from Theorem 1.1 in \cite[Theorem 1.1]{JohnsonMillson}, which says that $H_1(Sp_6^0(\mathbb{Z});\mathbb{Z}) \cong \mathbb{Z}/4$. 

Part (iii) follows by using the $E_2$-algebra map again and Theorem \ref{thm: 6.4}. 

For part (iv) we use Theorem \hyperref[theorem B]{B}, Part (i), to get that all the stabilization maps $\sigma_{\epsilon} \cdot - : H_1(Sp_{2(g-1)}^{\delta-\epsilon}(\mathbb{Z});\mathbb{Z}) \rightarrow H_1(Sp_{2g}^{\delta}(\mathbb{Z});\mathbb{Z})$ are surjective for $g \geq 4$. 
(The proof of part (i) of Theorem \hyperref[theorem B]{B} is independent of the computations in the Appendix.)

By \cite[Theorem 1.1]{JohnsonMillson} the stable first homology group of the quadratic symplectic groups of Arf invariant 0 is $\mathbb{Z}/4$. 
The stable first homology group of the quadratic symplectic groups of Arf invariant 1 must be the same by homological stability using Theorem \ref{theorem stab 1}. 
Thus, $H_1(Sp_{6}^{1}(\mathbb{Z});\mathbb{Z})$ surjects onto $\mathbb{Z}/4$. 
Finally, by a similar reasoning to the one at the beginning of this proof we get that $H_1(\Gamma_{3,1}^{1/2}[1];\mathbb{Z}) \cong \mathbb{Z}/4$ surjects onto $H_1(Sp_{6}^{1}(\mathbb{Z});\mathbb{Z})$.
The expression for the generator follows from Theorem \ref{thm: 6.4} and the $E_2$-algebra map. 
\end{proof}

\begin{rem}
All the computations of Section \ref{appendix symplectic} are consistent with the ones of \cite[Appendix A]{krannichmcg}. 
\end{rem}

\bibliographystyle{amsalpha}
\bibliography{bibliography}

\end{document}